\documentclass[10pt]{article}

\usepackage{amsmath,amsthm,amssymb}
\usepackage{caption,graphicx,subfig}
\usepackage[text={6.5in,9.25in},centering]{geometry}
\usepackage{paralist}
\usepackage[comma,square,sort&compress,numbers]{natbib}
\usepackage{wasysym}
\usepackage{xcolor}
\usepackage{empheq,etoolbox}
\usepackage[utf8]{inputenc}

\newcommand{\new}[1]{\textcolor{black}{#1}}

\setcaptionmargin{0.25in}

\makeatletter\@addtoreset{equation}{section}\makeatother

\newtheorem{thm}{Theorem}
\newtheorem{lem}{Lemma}[section]
\newtheorem{hyp}{Hypothesis}

\theoremstyle{definition}

\newcommand{\R}{\mathbb{R}}

\newenvironment{Acknowledgment}%
{\begin{trivlist}\item[]\textbf{Acknowledgments }}{\end{trivlist}}

\begin{document}

\title{Snaking bifurcations of localized patterns on ring lattices}

\author{
Moyi Tian\thanks{Division of Applied Mathematics, Brown University, Providence, RI, 02906}
\and
Jason J. Bramburger\thanks{Department of Applied Mathematics, University of Washington, Seattle, WA, 98105}
\and
Bj\"orn Sandstede\footnotemark[1]
}

\date{}
\maketitle

\begin{abstract}
We study the structure of stationary patterns in bistable lattice dynamical systems posed on rings with a symmetric coupling structure in the regime of small coupling strength. We show that sparse coupling (for instance, nearest-neighbour or next-nearest-neighbour coupling) and all-to-all coupling lead to significantly different solution branches. In particular, sparse coupling leads to snaking branches with many saddle-node bifurcations, whilst all-to-all coupling leads to branches with \new{six} saddle nodes, regardless of the size of the number of nodes in the graph.
\end{abstract}


\section{Introduction}

In this paper, we study the structure of stationary patterns in bistable lattice systems. More precisely, we take a ring of $N \geq 3$ identical nodes and consider the lattice dynamical system
\begin{equation}\label{LDSI}
\dot{u}_n = d(\Delta_m U)_n + f(u_n,\mu), \quad
1 \leq n \leq N, \quad
\new{U = (u_1,u_2\dots,u_N) \in \mathbb{R}^N}
\end{equation}
on this ring for small coupling strength $0<d\ll1$, where $f(u,\mu)$ is a bistable nonlinearity, and the coupling operator $\Delta_m$ denotes the symmetric $m$-nearest-neighbour connections on the ring \new{that we describe in more detail below. We refer to Figure~\ref{fig:Ring} for an illustration of typical nonlinearities and the coupling structures we consider in this paper. Throughout this paper, we will use the notation $U=(u_1,u_2\dots,u_N)\in\mathbb{R}^N$ and will always take all indices in the set $\{1,\dots,N\}$ modulo $N$ so that, for instance, $u_{N+1} = u_1$ and $u_{-1} = u_N$. With these conventions}, the coupling operator in the case $m=1$ is the usual discrete diffusive (nearest-neighbour) operator
\[
(\Delta_1 U)_n = u_{n+1} + u_{n-1} - 2u_n,
\]
the case $m=2$ results in the next-nearest-neighbour coupling
\[
(\Delta_2 U)_n = u_{n+2} + u_{n+1} + u_{n-1} + u_{n-2} - 4u_n,
\]
and the general case is given by
\[
(\Delta_m U)_n = 
\left\{ \begin{array}{lcl}
-(2m+1) u_n + \sum_{j=-m}^m u_{n+j}, & \quad & 1\leq m<\lfloor\frac{N}{2}\rfloor \\
-N u_n + \sum_{j=1}^N u_j, && m=\lfloor\frac{N}{2}\rfloor.
\end{array}\right.
\]
We refer to the case $m = \lfloor\frac N 2\rfloor$ as all-to-all coupling as each element on the ring is connected to all other elements.

\begin{figure}
\center
\includegraphics[scale=0.9]{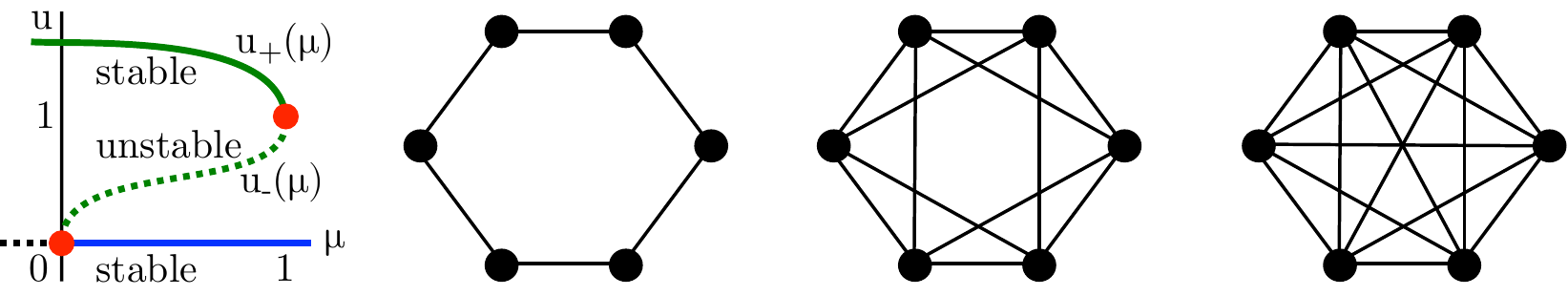} 
\caption{The left panel illustrates the zero set $\{(u,\mu): f(u,\mu)=0\}$ of a typical bistable nonlinearity $f(u,\mu)$ \new{with two stable states at $u=0$ and $u=u_+(\mu)$ and an unstable state at $u=u_-(\mu)$}. The three rightmost panels contain graphs that consist of $N=6$ identical nodes with (from left to right) nearest neighbour, next-nearest neighbour, and all-to-all coupling, respectively.}
\label{fig:Ring}
\end{figure} 

Our goal is to understand how the arrangement of stationary patterns in the $(u,\mu)$ configuration space depends on the interaction length $m$ of the coupling on the ring. Steady states of (\ref{LDSI}) correspond to solutions to the system
\begin{equation}\label{LDSS}
\mathcal{F}(U,\mu,d) = 0
\end{equation}
where
\[
\mathcal{F}: \R^N\times\R\times\R \longrightarrow \R^N, \quad
(U,\mu,d) \longmapsto \mathcal{F}(U,\mu,d), \quad
\mathcal{F}(U,\mu,d)_n = d(\Delta_m U)_n + f(u_n,\mu).
\]
We focus on the case $0<d\ll1$, which allows us to exploit \new{the anti-continuum limit $d=0$ when} the system (\ref{LDSS}) is uncoupled. Setting $d=0$ and choosing $\mu\in(0,1)$, we select the solution of (\ref{LDSS}) for which \new{$u_1$ lies on the upper stable branch $u=u_+(\mu)$ and the remaining nodes $u_j$ with $j\neq1$ lie on the lower stable branch $u=0$ of the zero set of $f(u,\mu)$; see Figure~\ref{fig:Ring} for an illustration of the zero set of $f(u,\mu)$}. We are then interested in describing the connected component of the solution set $\{(U,\mu,d): \mathcal{F}(U,\mu,d)=0\}$ that this solution belongs to and understand how this connected component changes as the interaction range $m$ varies.

\begin{figure} 
\center
\includegraphics[width = \textwidth,height = 0.4\textwidth]{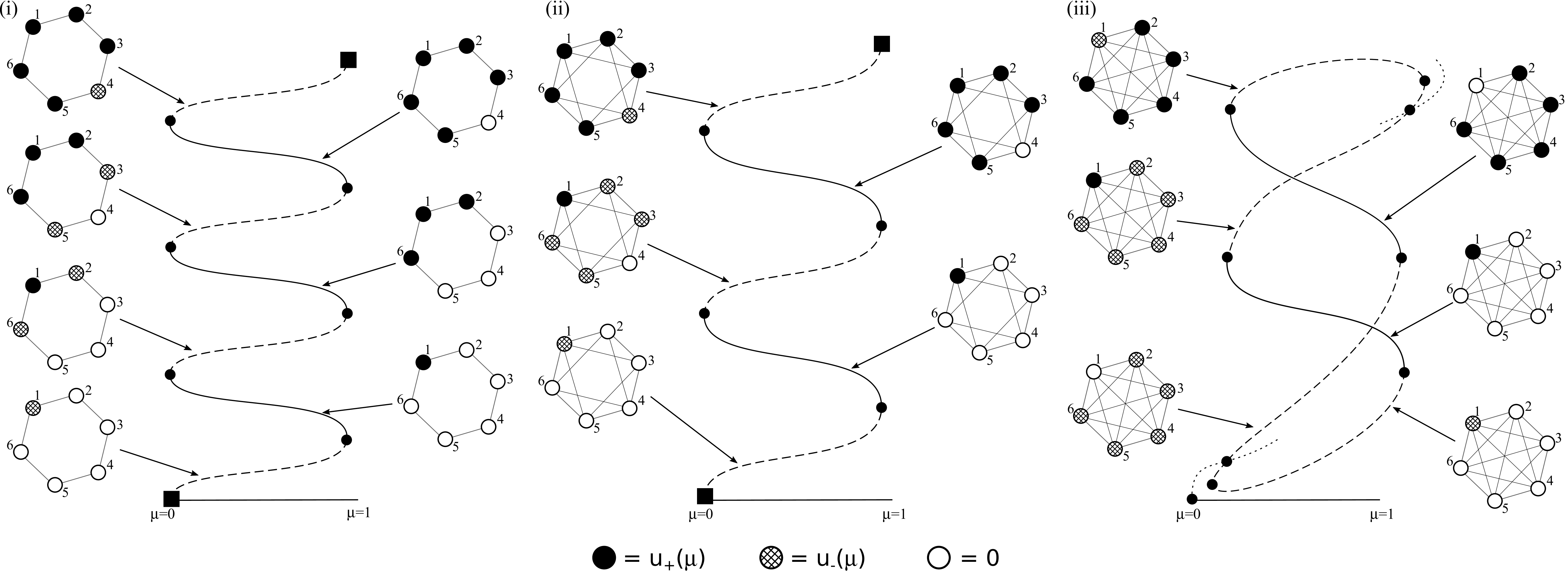} 
\caption{Different bifurcation curves depending on the number of symmetric connections over a six element ring. (i) Nearest-neighbour connections lead to a typical snaking bifurcation diagram, as is proven in Theorem~\ref{thm:Sparse}, (ii) Next-nearest-neighbour connections on a six element ring introduce a further symmetry for which the elements indexed by 2,3,5,6 \new{bifurcate} together near $\mu = 0$, as proven in Theorem~\ref{thm:6,2}. (iii) In the case of all-to-all coupling, the system is invariant with respect to any permutation of the nodes on the ring and so the bifurcation curves form a small closed curve that bifurcates from the homogeneous state (dotted curve) near the origin, as proven in Theorem~\ref{thm:All2All}. Squares marking the end of the curves in panels (i) and (ii) represent the set of exceptional bifurcations which are not proven in this work.}
\label{fig:Intro}
\end{figure} 

\new{For $m=1$ and each fixed value of $0<d\ll1$, we will show that solutions are arranged in a ``snaking" pattern (see Figure~\ref{fig:Intro}(i)): starting with a configuration where one node is set to $u_+(\mu)$ and all other nodes are set to zero, we show that the solution branch exhibits fold bifurcations at which the number of nodes that are close to $u_+(\mu)$ increases until all nodes are equal to $u_+(\mu)$. Snaking has first been conjectured in \cite{Pomeau} for spatially extended systems modelled by partial differential equations (PDEs) on the real line and was first investigated rigorously in the seminal work \cite{Woods}; see also \citep{Coullet-PRL} for related work. Since then, snaking for PDEs in one and two space dimensions has been analyzed extensively, and we refer, for instance, to \citep{Avitabile,Beck,burke-2006-pre,burke-2007-pla,Chapman,Kozyreff2006,Lloyd} and the reviews \citep{review1,review2}. Many of these investigations were motivated by the observation of snaking patterns in experiments and models, including in ferrofluids \cite{Ferrofluid, LloydRichter, Richter}, optical systems \cite{Chong, gomila2007, firth2007}, and vegetation models \cite{Meron_Survey}, to name but a few.}
Similarly complex bifurcation structures of localized solutions have also been observed in spatially discrete systems posed on integer lattices in \cite{Chong,Chong2,Susanto2,McCullen,Papangelo,Susanto1,Taylor,Yulin,Yulin2} and were explained in part by \cite{Isolas,Bramburger,Bramburger2}. While the latter works have explained the bifurcation structure of localized solutions on ``regular'' graphs, such as the integer lattices, little is known about how graph structure and connection topology influence the connections of localized solutions in parameter space.

\new{Motivated by these previous investigations, we explore here the role of different finite interaction lengths and the effect of boundaries on snaking patterns in lattice systems. Focusing on ring lattices allows us to address these questions together. We} will provide analytical results for two distinct coupling regimes, namely sparse coupling (with nearest-neighbour and next-nearest-neighbour coupling) and all-to-all coupling. We will prove that the resulting bifurcation curves differ significantly as illustrated in Figure~\ref{fig:Intro}: sparse coupling leads to snaking curves with many fold bifurcations \new{(the number will increase as the number of nodes, $N$, increases)}, while all-to-all coupling will always result in a curve with \new{six saddle nodes and two symmetry-breaking bifurcations near $\mu = 0$ and $\mu = 1$ from the branch of equilibria at which all nodes are equal to the unstable state $u_-(\mu)$}. We also provide illustrative examples in the case of almost all-to-all coupling where additional symmetries in the system result in entirely unique bifurcation curves that would be difficult to predict. We conjecture that sparse coupling will lead to snaking curves that exhibit $\lfloor\frac N 2\rfloor$ saddle nodes for each $m$ with \new{$1\leq m\leq\lfloor\frac N 2\rfloor-2$} (thus excluding almost all-to-all coupling).


\section{Main results}\label{sec:Results}

Recall that we are interested in the solution structure of the system
\begin{equation}\label{LDS}
\dot{u}_n = d(\Delta_m U)_n + f(u_n,\mu), \quad 1 \leq n \leq N.
\end{equation}
First, note that system \eqref{LDS} is equivariant with respect to the dihedral symmetry group $D_N$, generated by the actions
\begin{equation}
	\begin{split}
		\zeta (u_1,u_2,\dots,u_N) &= (u_2,\dots,u_N,u_1) \\
		\kappa(u_1,u_2,u_3,\dots,u_{N-1},u_N) &= (u_1,u_N,u_{N-1},\dots,u_3,u_2)
	\end{split}
\end{equation}
which define rotations and flips of the ring, respectively. Clearly $\zeta^N = 1$ and $\kappa^2 = 1$, where we use $1$ to denote the identity element that acts trivially on vectors in $\R^N$. Our interest in what follows will lie in solutions of \eqref{LDS} that are invariant with respect to the action of $\kappa$. Notice that the action of $\kappa$ on rings with $N$ even leaves exactly two elements fixed, while when $N$ is odd only the element at index $1$ is fixed. We refer the reader to Figure~\ref{fig:Flip} for a demonstration of these two cases with $N = 5,6$. Finally, in the case of all-to-all coupling, system \eqref{LDS} is equivariant with respect to the symmetric group \new{$S_N$} of permutations on $N$ elements, which is a strictly larger class of symmetries than those for all other classes of coupling functions considered in this work. 

\begin{figure} 
\center
\includegraphics[width = 0.5\textwidth]{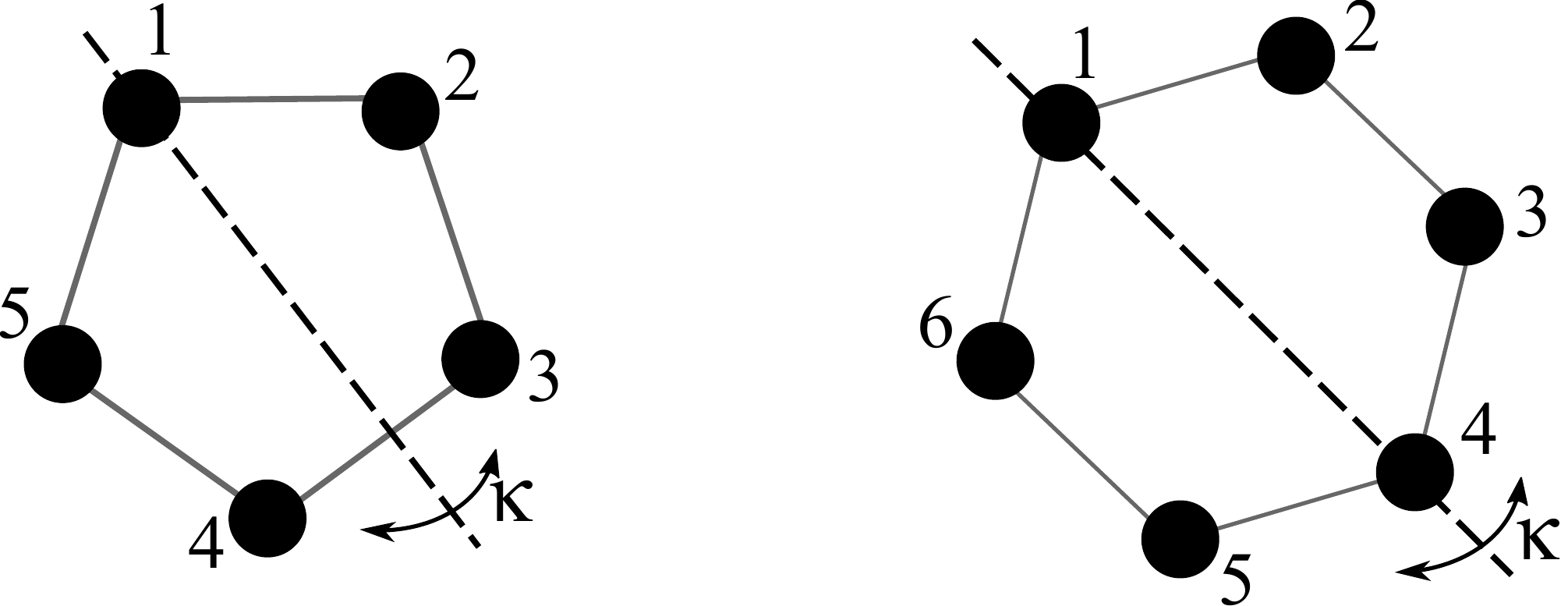} 
\caption{Action of the flip $\kappa$ on odd and even element rings. On the left is a ring of $N = 5$ elements, for which the flip leaves only the element at index $1$ fixed, while the right presents a ring of $N=6$ elements, for which the flip leaves the elements at index $1$ and $4$ fixed.}
\label{fig:Flip}
\end{figure} 

Here the function $f$ is assumed to be bistable, satisfying the same properties as in \cite{Bramburger2}. We summarize these properties in the following hypothesis; see also Figure~\ref{fig:Ring}.

\begin{hyp} \label{h1}
The function \new{$f:\mathbb{R}\times\mathbb{R}\to\mathbb{R}$} is smooth and satisfies the following:
\begin{compactenum}[(i)]
\item The function $f$ is odd in $u$ so that $f(-u,\mu)=-f(u,\mu)$ for all $(u,\mu)$.
\item The set of roots of $f(u,\mu)$ is as shown in the left panel of Figure~\ref{fig:Ring}. In particular, for each $\mu\in(0,1)$, the function $f(u,\mu)$ has exactly three nonnegative zeros, namely $u=0$ and $u=u_\pm(\mu)$ with $0<u_-(\mu)<u_+(\mu)$, and these satisfy $f_u(0,\mu),f_u(u_+(\mu),\mu)<0<f_u(u_-(\mu),\mu)$.
\item At $\mu=0$, the zeros $u=0$ and $u=\pm u_-(\mu)$ collide in a generic subcritical pitchfork bifurcation.
\item At $\mu=1$, the zeros $u=u_\pm(\mu)$ collide in a generic saddle-node bifurcation \new{at $u=1$}.
\end{compactenum}
\end{hyp}

\new{Our assumption (iv) that the saddle-node bifurcation at $\mu=1$ occurs at $u=1$ is for convenience only as this can always be achieved through an appropriate change of variables. We also note} that all of our analysis can be applied to bistable functions for which the bifurcation at $\mu = 0$ or $\mu = 1$ is instead a transcritical bifurcation, as was shown to be true in \cite{Bramburger2}. To precisely state our results we instead focus on functions satisfying Hypothesis~\ref{h1} with the prototypical example being the cubic-quintic nonlinearity
\begin{equation}
	f(u,\mu) = -\mu u + 2u^3 - u^5. 
\end{equation}

\new{This type of nonlinearity appears naturally from the study of optical solitons in cubic-quintic nonlinear Schr\"odinger lattices \cite{Chong, Chong2} and it serves as a simple solid example for analysis that gives rise to snaking localized patterns.}
Throughout this section we detail our results on localized pattern formation in the lattice system \eqref{LDS} with functions $f$ satisfying Hypothesis~\ref{h1} based on the value of $m$. \new{In the statement of our proofs we will adopt the notation $\mathcal{B}_\delta(X)$ to denote an open set of all points within distance $\delta > 0$ from the set $X$.} All proofs will be left to Section~\ref{sec:Proofs}.


\subsection{Sparse coupling}\label{subsec:Sparse}


We begin with the case of ``sparse'' coupling, i.e. $m \ll \lfloor N/2 \rfloor$. In particular, we will focus on the cases of $m = 1,2$ to show that localized solutions grow around the ring in the form of a snaking bifurcation curve with the region of activation growing symmetrically around the ring as one ascends the diagram, as is illustrated in panel (i) of Figure~\ref{fig:Intro}. Since we are interested in solutions which are invariant with respect to the action of $\kappa$, we can restrict our attention to the index set 
\begin{equation}\label{IndexSet}
	I = \bigg\{1\leq n\leq \bigg\lfloor \frac N 2\bigg\rfloor + 1\bigg\}
\end{equation}
for any fixed $N \geq 2$. Indeed, an element $U\in\R^N$ satisfying $\kappa U = U$ is uniquely identified by elements with indices belonging to $I$. The reader is referred to Figure~\ref{fig:Flip}, where we see that when $N = 5$ we have $I = \{1,2,3\}$ and the elements at indices $4,5$ are identical to those at $3,2$, respectively. Similarly, from Figure~\ref{fig:Flip}, when $N = 6$ we have $I = \{1,2,3,4\}$ and the elements at indices $5,6$ are identical to those at $3,2$, respectively. 

Hypothesis~\ref{h1} implies that any nonnegative solution of \new{\eqref{LDSS}} with $d = 0$ must have $u_n \in \{0,u_\pm(\mu)\}$ when $\mu \in [0,1]$. \new{We use $k$ to specify nodes that are activated at $u_+(\mu)$.} Then, for each $0 \leq \mu \leq 1$ and $k \in I$ we define the elements $\bar{U}^{(k)}(\mu) = \{\bar{u}^{(k)}_n(\mu)\}_{n\in I}$ and $\bar{V}^{(k)}(\mu)= \{\bar{v}^{(k)}_n(\mu)\}_{n\in I}$ by
\begin{equation}\label{uBar}
	\bar{u}^{(k)}_n(\mu) = \begin{cases}
		u_+(\mu) & 1 \leq n \leq k \\
		0 & n> k
	\end{cases} 
\end{equation}
and
\begin{equation}\label{vBar}
	\bar{v}^{(k)}_n(\mu) = \begin{cases}
		u_+(\mu) & 1\leq n < k \\
		u_-(\mu) & n = k \\
		0 & n > k
	\end{cases}
\end{equation}
for all $n \in I$ and $\mu \in [0,1]$. Notice that when $k = \lfloor N/2 \rfloor + 1$ we have that $\bar U^{(k)}$ is a uniform state with all entries given by $u_+(\mu)$. From the discussion above, we have that these elements can be extended to $\kappa$-invariant solutions of \eqref{LDSS} by having $u_n$ for $n \notin I$ be defined by the relation $\kappa U = U$. 

The elements \eqref{uBar} and \eqref{vBar} are pairwise distinct when $\mu \in (0,1)$, but Hypothesis~\ref{h1} gives that $\lim_{\mu \to 0^+} u_-(\mu) = 0$ and $\lim_{\mu \to 1_-} u_-(\mu) = u_+(1)$, and so the patterns \eqref{uBar} and \eqref{vBar} satisfy
\begin{equation}
	\begin{split}
		\bar{U}^{(k-1)}(0) &= \bar{V}^{(k)}(0), \quad k = 2,\dots,\bigg\lfloor \frac N 2 \bigg\rfloor + 1 \\
		\bar{U}^{(k)}(1) &= \bar{V}^{(k)}(1), \quad k = 1,\dots,\bigg\lfloor \frac N 2 \bigg\rfloor
	\end{split}
\end{equation}  
at the parameter boundaries $\mu = 0,1$. We can therefore define the connected set
\begin{equation}
	\Gamma_\mathrm{sparse} := \bigcup_{k \in I}\bigcup_{0\leq \mu \leq 1} \{(\bar{U}^{(k)}(\mu),\mu),(\bar{V}^{(k)}(\mu),\mu)\}\subset \R^N\times[0,1],
\end{equation}
which represents the union of the curves traced out for $\mu \in [0,1]$ by the patterns \eqref{uBar} and \eqref{vBar} of \eqref{LDSS} when $d = 0$. We will also define the exceptional set $\mathcal{E}$, given by
\begin{equation}\label{Exceptional}
	\mathcal{E} := \{\bar{V}^{(1)}(0)\}\cup\{\bar{V}^{(\lfloor \frac N 2 \rfloor + 1)}(1)\},
\end{equation} 
representing the endpoints of the curve $\Gamma_\mathrm{sparse}$. We present the following theorem, for which Figure~\ref{fig:Intro}(i) provides an illustration of the results for $(N,m) = (6,1)$. The proofs for $m = 1$ are left to \S\ref{subsec:NNProof} and $m = 2$ are left to \S\ref{subsec:NNNProof}. 

\begin{thm}\label{thm:Sparse}
	Assume that $f$ satisfies Hypothesis~\ref{h1}. If $N \geq 4$ and $m = 1$ or $N \geq 7$ and $m = 2$, then for each $\delta_* > 0$ there exists $d_* > 0$ such that \new{for each $0<d<d_*$} the set \new{$\mathcal{B}_{\delta_*}(\Gamma_\mathrm{sparse})\setminus \mathcal{B}_{2\delta_*}(\mathcal{E})$} contains a unique, nonempty, continuous branch of $\kappa$-symmetric solutions of the steady-state system \eqref{LDSS}. Furthermore, this branch is smooth and $C^1$-close to $\Gamma_\mathrm{sparse}$ for each $d$, depends smoothly on $d$, and its limit as $d \to 0^+$ is contained in $\Gamma_\mathrm{sparse}$.
\end{thm}


\subsection{Almost all-to-all coupling}\label{subsec:Almost}


Notice that in the statement of Theorem~\ref{thm:Sparse} we were required to take $N \geq 7$ when $m = 2$ to obtain the snaking bifurcation curves of localized solutions similar to those with $m = 1$. This is because when $N$ is even and $m = \lfloor \frac N 2 \rfloor - 1$, new symmetries are introduced into the model. This has the effect that the resulting bifurcation diagram is markedly different than the usual snaking diagram for sparse coupling, while also being distinct from the fully symmetric case of all-to-all coupling covered in the following subsection. It appears that many of these bifurcations curves must be understood on a case-by-case basis for varying $N$, and so instead of attempting to exhaustively document all of these cases, we opt to illustrate with two specific examples. The cases detailed here will take $N = 6,8$, representing the smallest ring sizes where these atypical bifurcation diagrams can be observed.

Let us begin with $N = 6$. Much of the bifurcation structure in the case when $(N,m) = (6,2)$ is similar to that of the case when $N \geq 7$ and $m = 2$, with the following exception: the connection from the continued solutions of $\bar U^{(1)}(\mu)$ to $\bar V^{(2)}(\mu)$, defined in \eqref{uBar} and \eqref{vBar}, respectively, near $(d,\mu) = (0,0)$ is not present. This connection is replaced by a connection from the continued solutions of $\bar U^{(1)}(\mu)$ to the branch continued from the solution $\bar{W}^{(23)}(\mu)$, given by
\begin{equation}
	\bar{w}^{(23)}_n = \begin{cases}
		u_+(\mu) & n = 1 \\
		u_-(\mu) & n = 2,3 \\
		0 & n = 4
	\end{cases}
\end{equation}
for each $\mu \in [0,1]$. Since the solution $\bar{W}^{(23)}(\mu)$ satisfies 
\[
\lim_{\mu \to 0^+} \bar{W}^{(23)}(\mu) = \bar U^{(1)}(0) \quad\mbox{and}\quad
\lim_{\mu \to 1^-} \bar{W}^{(23)}(\mu) = \bar U^{(3)}(1),
\]
we may define the connected set 
\begin{equation}
	\Gamma_{6,2} := \bigcup_{0\leq\mu\leq 1}\{(\bar V^{(1)},\mu ),(\bar U^{(1)},\mu ), (\bar W^{(23)},\mu ), (\bar U^{(3)},\mu ),(\bar V^{(4)},\mu )\}\subset\R^6\times[0,1].
\end{equation} 
Using again $\mathcal{E}$ defined in \eqref{Exceptional}, this leads to the following theorem whose results are presented visually in Figure~\ref{fig:Intro}(ii). The proof is left to \S\ref{subsec:AlmostProof}.

\begin{thm}\label{thm:6,2} 
	Assume that $f$ satisfies Hypothesis~\ref{h1} and that $(N,m) = (6,2)$. Then for each $\delta_* > 0$ there exists $d_* > 0$ such that for each $0 < d < d_*$ the set \new{$\mathcal{B}_{\delta_*}(\Gamma_{6,2})\setminus \mathcal{B}_{2\delta_*}(\mathcal{E})$} contains a nonempty, continuous branch of $\kappa$-symmetric solutions of the steady-state system \eqref{LDSS}. Furthermore, this branch is smooth and $C^1$-close to $\Gamma_{6,2}$ for each $d$, depends smoothly on $d$, and its limit as $d \to 0^+$ is contained in $\Gamma_{6,2}$.
\end{thm}

\begin{figure} 
\center
\includegraphics[width = 0.6\textwidth]{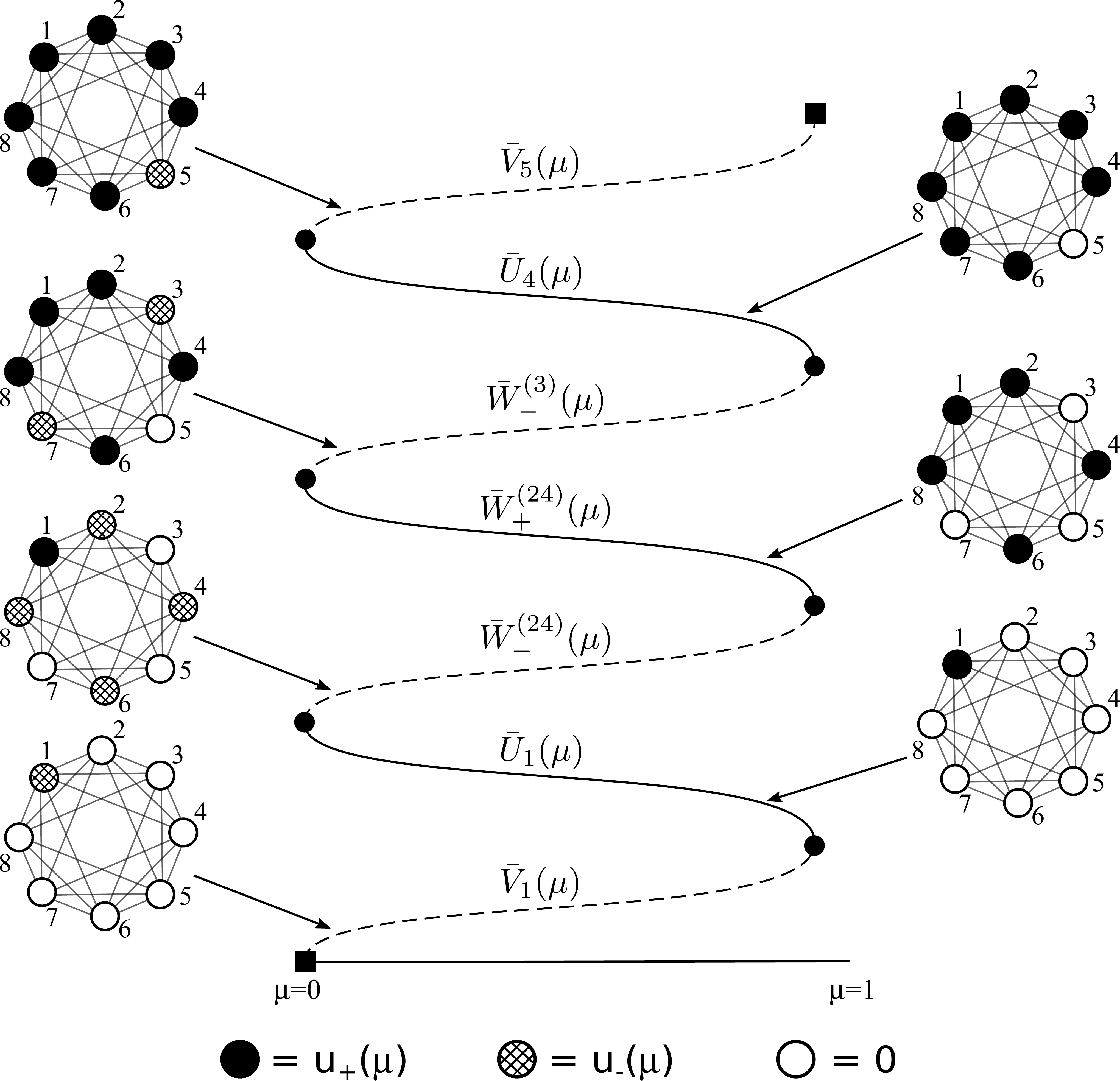} 
\caption{The bifurcation diagram on a ring with $(N,m) = (8,3)$ for small $d > 0$. The number of neighbour connections induces an added symmetry, thus leading to the above bifurcation diagram which bears little resemblance to that of the case $(N,m) = (8,1),(8,2)$ featured in Theorem~\ref{thm:Sparse}. Squares marking the end of the curve represent the set of exceptional bifurcations, continued from $\mathcal{E}$, which are not proven in this work.}
\label{fig:Special_Snake_8}
\end{figure} 

Turning now to the case $(N,m) = (8,3)$, we are required to define two more patterns which correspond to $\kappa$-invariant solutions of \new{\eqref{LDSS}} with $d = 0$. Consider the elements $\bar W^{(24)}_\pm(\mu)$ and $W^{(3)}_-(\mu)$, given by
\begin{equation}
	\bar{w}^{(24)}_{\pm,n}(\mu) = \begin{cases}
		u_+(\mu) & n = 1 \\
		u_\pm(\mu) & n = 2,4 \\
		0 & n = 3,5
	\end{cases} 
\end{equation}
and
\begin{equation}
	\bar{w}^{(3)}_{-,n}(\mu) = \begin{cases}
		u_+(\mu) & n = 1,2,4 \\
		u_-(\mu) & n = 3 \\
		0 & n = 5
	\end{cases}
\end{equation}
respectively, for each $\mu \in [0,1]$. Here the values in the superscript detail which terms are at $u_-(\mu)$ when the subscript is $-$, while the element $\bar W^{(24)}_+(\mu)$ connects the two elements with $-$ subscripts. We refer the reader to panels (ii) and (iii) of Figure~\ref{fig:Almost_All2All} below for visual depictions of these elements extended by $\kappa$ symmetry to the entire ring. Notice that we have the following connections:
\begin{equation}
	\begin{split}
		\lim_{\mu \to 0^+} \bar W^{(24)}_-(\mu) &= \lim_{\mu \to 0^+} \bar U^{(1)}(\mu) \\ 
		\lim_{\mu \to 1^-} \bar W^{(24)}_+(\mu) &= \lim_{\mu \to 1^-} \bar W^{(\new{24})}_-(\mu) \\ \new{\lim_{\mu \to 0^+} \bar W^{(3)}_-(\mu)} &\new{= \lim_{\mu \to 0^+} \bar W^{(\new{24})}_+(\mu) }\\
		\lim_{\mu \to \new{1^-}} \bar W^{(3)}_-(\mu) &= \lim_{\mu \to \new{1^-}} \bar U^{(4)}(\mu).
	\end{split}
\end{equation}  
These connections allow one to define the connected set
\begin{equation}
	\Gamma_{8,3} := \bigcup_{0\leq\mu\leq 1}\{(\bar V^{(1)},\mu ),(\bar U^{(1)},\mu ), (\bar W^{(24)}_-,\mu ), (\bar W^{(24)}_+,\mu ), (\bar W_-^{(3)},\mu ), (\bar U^{(4)},\mu ), (\bar V^{(5)},\mu )\}\subset\R^8\times[0,1].
\end{equation}
This leads to the following theorem whose results are represented visually in Figure~\ref{fig:Special_Snake_8}. The proof is again left to \S\ref{subsec:AlmostProof}.

\begin{thm}\label{thm:8,3}
	Assume that $f$ satisfies Hypothesis~\ref{h1} and that $(N,m) = (8,3)$. Then for each $\delta_* > 0$ there exists $d_* > 0$ such that for each $0 < d < d_*$ the set \new{$\mathcal{B}_{\delta_*}(\Gamma_{8,3})\setminus \mathcal{B}_{2\delta_*}(\mathcal{E})$} contains a nonempty, continuous branch of $\kappa$-symmetric solutions of the steady-state system \eqref{LDSS}. Furthermore, this branch is smooth and $C^1$-close to $\Gamma_{8,3}$ for each $d$, depends smoothly on $d$, and its limit as $d \to 0^+$ is contained in $\Gamma_{8,3}$.
\end{thm}


\subsection{All-to-all coupling}\label{subsec:All}


Let us now consider the case of all-to-all coupling, i.e. $m = \lfloor \frac N 2 \rfloor$. In this case \eqref{LDSS} becomes
\begin{equation}\label{AllLDS}
	d\sum_{j = 1}^N (u_j - u_n) + f(u_n,\mu) = 0,
\end{equation}
so that every element is coupled to every other element, \new{thus representing a complete graph with $N$ vertices. The symmetry group  associated with \eqref{AllLDS} is therefore given by the group \new{$S_N$} of all permutations of the vectors $U = (u_1,u_2,\dots,u_N)\in\R^N$, which is larger than the dihedral group $D_N$. For each $k = 1,\dots,\new{\lfloor N/2 \rfloor}$, we will construct solutions whose first $k$ elements are identical and whose last $N-k$ elements are identical: such solutions admit the  isotropy group $S_k\times S_{N-k}$, defined by the set of all permutations of the first $k$ elements of a vector in $\R^N$ together with all permutations of the last $N-k$ elements of the vector. Note that the fixed point space of $S_k\times S_{N-k}$ in $\R^N$ is the two-dimensional space of all vectors whose first $k$ elements are identical and whose last $N-k$ elements are identical.}

\new{Generic symmetry-breaking bifurcations from homogeneous equilibria with isotropy group $S_N$ to equilibria with isotropy $S_k\times S_{N-k}$ were investigated in \cite{elmhirst, dias} in the context of sympathic speciation (see, for instance, \cite{cohen,stewart} for further background). Genericity requires in particular that the linearization at a homogeneous equilibrium at the bifurcation point has a zero eigenvalue in precisely one of the two complementary $S_N$-invariant subspaces given by $\{U\in\mathbb{R}^N:\,u_1=\ldots=u_N\}$ and $\{U\in\mathbb{R}^N:\,u_1+\ldots+u_N=0\}$ but not in both simultaneously; see \cite[\S2]{elmhirst}. Unfortunately, the bifurcations occurring in (\ref{AllLDS}) at $(d,\mu)=(0,0)$ and $(d,\mu)=(0,1)$ are not generic as the zero eigenvalue has multiplicity $N$, and we therefore need to analyse the resulting bifurcations occurring in (\ref{AllLDS}) here.}

\begin{figure}
\center
\includegraphics[scale=0.25]{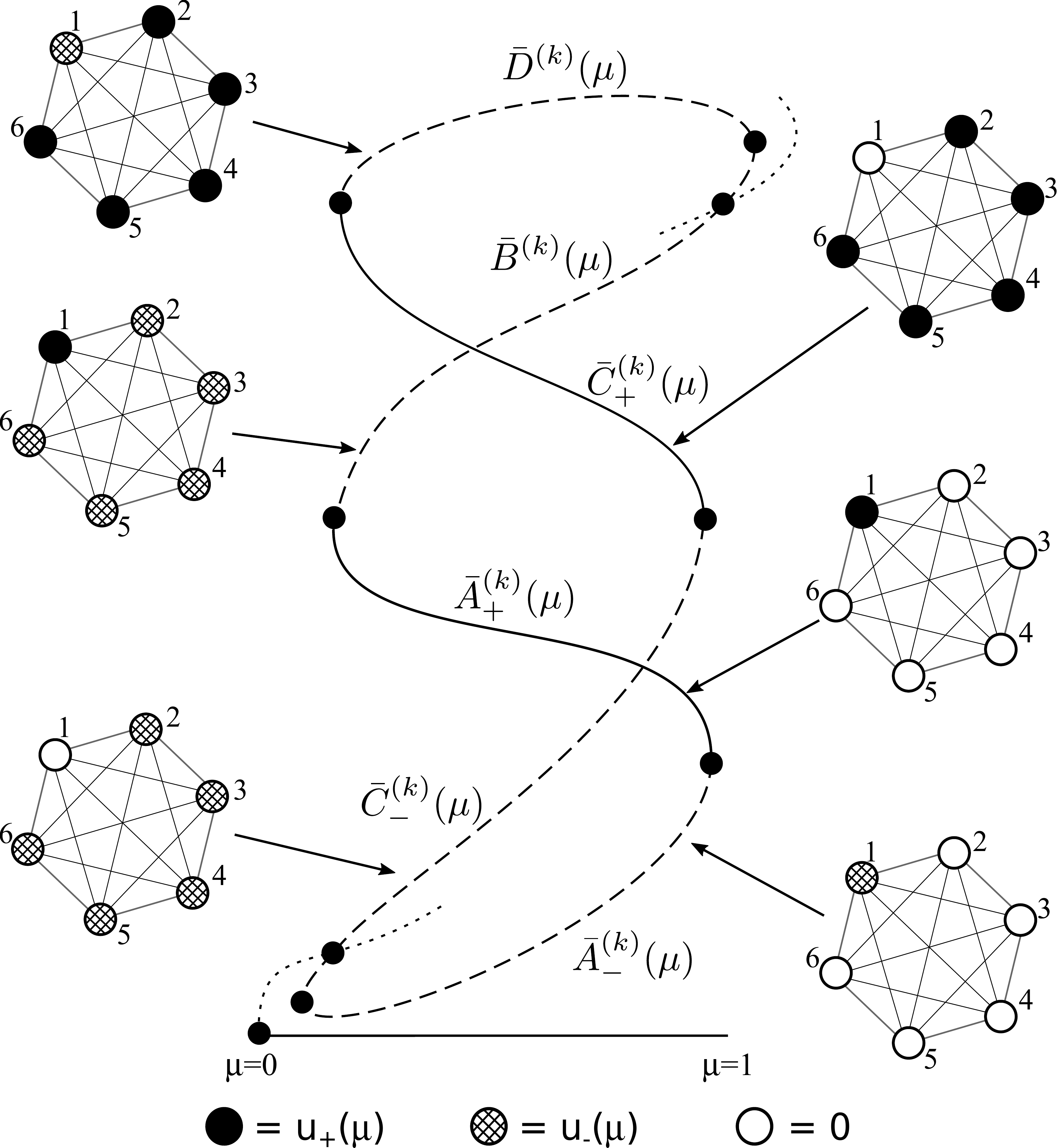} 
\caption{\new{We illustrate the solution branch for symmetric all-to-all coupling and the different patterns along the branch. The dotted curve corresponds to the branch of $S_N$-symmetric equilibria at which all nodes are equal to the unstable state $u_-(\mu)$.}}
\label{fig:All2AllABCD}
\end{figure} 

For each $k = 1,\dots,\new{\lfloor N/2 \rfloor}$ and each $0\leq\mu\leq1$, we define the \new{$S_k\times S_{N-k}$}-invariant vectors $\bar A_\pm^{(k)}(\mu) = \{\bar a_{\pm,n}^{(k)}(\mu)\}$, $\bar B^{(k)}(\mu) = \{\bar b_{n}^{(k)}(\mu)\}$, $\bar C_\pm^{(k)}(\mu) = \{\bar c_{\pm,n}^{(k)}(\mu)\}$, and $\bar D^{(k)}(\mu) = \{\bar d_{n}^{(k)}(\mu)\}$ by  
\begin{equation}\label{aBar}
	\bar{a}^{(k)}_{\pm,n}(\mu) = \begin{cases}
		u_\pm(\mu) & 1 \leq n \leq k \\
		0 & n> k
	\end{cases} 
\end{equation}
\begin{equation}\label{bBar}
	\bar{b}^{(k)}_n(\mu) = \begin{cases}
		u_+(\mu) & 1\leq n \leq k \\
		u_-(\mu) & n > k 
	\end{cases}
\end{equation}
\begin{equation}\label{cBar}
	\bar{c}^{(k)}_{\pm,n}(\mu) = \begin{cases}
		0 & 1 \leq n \leq k \\
		u_\pm(\mu) & n> k
	\end{cases} 
\end{equation}
and
\begin{equation}\label{dBar}
	\bar{d}^{(k)}_n(\mu) = \begin{cases}
		u_-(\mu) & 1\leq n \leq k \\
		u_+(\mu) & n > k;
	\end{cases}
\end{equation}
\new{see Figure~\ref{fig:All2AllABCD} for an illustration}. From Hypothesis~\ref{h1} we have that for any $k$ the vectors \eqref{aBar}--\eqref{dBar} are \new{$S_k\times S_{N-k}$}-invariant solutions of \eqref{LDSS} when $d = 0$. Furthermore, we have the following connections:
\begin{equation}
	\begin{split}
	    \new{\lim_{\mu \to 0^+} \bar A_-^{(k)}(\mu)} &\new{= \lim_{\mu \to 0^+} \bar C_-^{(k)}(\mu)} \\
		\lim_{\mu \to 1^-} \bar A_-^{(k)}(\mu) &= \lim_{\mu \to 1^-} \bar A_+^{(k)}(\mu) \\ 
		\lim_{\mu \to 0^+} \bar A_+^{(k)}(\mu) &= \lim_{\mu \to 0^+} \bar B^{(k)}(\mu) \\
		\lim_{\mu \to 1^-} \bar B^{(k)}(\mu) &= \lim_{\mu \to 1^-} \bar D^{(k)}(\mu) \\ 
		\lim_{\mu \to 0^+} \bar C_+^{(k)}(\mu) &= \lim_{\mu \to 0^+} \bar D^{(k)}(\mu) \\
		\lim_{\mu \to 1^-} \bar C_-^{(k)}(\mu) &= \lim_{\mu \to 1^-} \bar C_+^{(k)}(\mu) 
	\end{split}
\end{equation}
for each $k$. We therefore introduce the connected curve
\begin{equation}
	\Gamma_\mathrm{all}^k := \bigcup_{0 \leq \mu \leq 1} \{(\bar A_\pm^{(k)}(\mu),\mu),(\bar B^{(k)}(\mu),\mu),(\bar C_\pm^{(k)}(\mu),\mu),(\bar D^{(k)}(\mu),\mu)\} \subset \R^N\times [0,1],
\end{equation}
for each $k = 1,\dots,\new{\lfloor N/2 \rfloor}$. We present the following result, whose proof is left to \S\ref{subsec:AllProof}. 

\begin{thm}\label{thm:All2All}
Assume that $f$ satisfies Hypothesis~\ref{h1}. For each $k = 1,\dots,\new{\lfloor N/2 \rfloor}$ and $\delta_* > 0$, there exists $d_* > 0$ such that for each $0 < d < d_*$ the set \new{$\mathcal{B}_{\delta_*}(\Gamma_\mathrm{all}^k)$ contains two unique, distinct, nonempty, and continuous branches of $S_k\times S_{N-k}$-symmetric solutions of the steady-state system \eqref{AllLDS}, which emerge from and terminate at the homogeneous branch of solutions to \eqref{LDSS} given by $u_n=u_-(\mu)$ for $n = 1,\dots, N$ at $\mu=\frac{Nd}{2}+\mathcal{O}(d^2)$ and $\mu=1-(\frac{Nd}{2})^2+\mathcal{O}(d^3)$, respectively.} Furthermore, \new{these branches are} smooth and $C^1$-close to $\Gamma_\mathrm{all}^k$ for each $d$, depends smoothly on $d$, and its limit as $d \to 0^+$ is contained in $\Gamma_\mathrm{all}^k$.
\end{thm}

We note that the results of Theorem~\ref{thm:All2All} do not include an exceptional set, meaning that the theorem completely characterizes the entire bifurcation curve of localized solutions in the presence of all-to-all coupling. We refer to \new{Figure~\ref{fig:All2AllABCD} for an illustration} of the result with a ring of $N = 6$ elements. Interestingly, we see that this curve of solutions does not originate at the origin $(U,\mu) = (0,0)$ in $\R^N\times[0,1]$ when $0 < d \ll 1$, but instead bifurcates from the homogeneous branch of solutions at a positive value of $\mu$. For more details on this bifurcation from the homogeneous branch, we refer the reader to Lemma~\ref{lem:all_ll} below.  


\section{Proofs}\label{sec:Proofs}


In this section we will prove the various theorems from the previous section detailing the existence and bifurcation structure of steady-state solutions to \eqref{LDS}. As discussed in the introduction, we can define the function
\begin{equation}\label{Function} 
\mathcal{F}: \;
\R^N\times\R\times\R\longrightarrow\R^N, \quad
(U,\mu,d) \longmapsto \mathcal{F}(U,\mu,d), \quad
\mathcal{F}(U,\mu,d)_n := d(\Delta_m U)_n + f(u_n,\mu),
\end{equation}
to see that the steady-state system
\begin{equation}\label{SteadyLDS}
	d(\Delta_m U)_n + f(u_n,\mu) = 0	
\end{equation}
corresponding to \eqref{LDS} is given by $\mathcal{F}(U,\mu,d)=0$. Note that $\mathcal{F}$ is smooth in its arguments, and upon taking $d = 0$, solving $\mathcal{F}(U,\mu,0) = 0$ with $\mu \in (0,1)$ reduces to taking $u_n \in \{0,\pm u_\pm(\mu)\}$ for each $n = 1,\dots N$ per Hypothesis~\ref{h1}.

From the assumption on the non-degeneracy of the roots of $f$, it follows that for $\mu$ belonging to any compact interval of $(0,1)$ these solutions may be continued regularly into $d > 0$ using the implicit function theorem. \new{The challenge is therefore to understand the cases where $\mu$ is close to zero or one. In these cases, the Jacobian of $\mathcal{F}$ at $d=0$ will have a null space of dimension at least one, and Lyapunov--Schmidt reduction allows us to reduce the equation $\mathcal{F}(U,\mu,d)=0$ to a reduced equation $\mathcal{F}^\mathrm{c}(U^\mathrm{c},\mu,d)=0$ defined on the null space. As we will see below, the reduced equation $\mathcal{F}^\mathrm{c}(U^\mathrm{c},\mu,d)=0$ is, to leading order, quasihomogeneous, and we now explain what this means and how it helps us in our analysis. Given exponents  $a_1,a_2,a_3\geq1$, and focusing for ease of notation on the case where $\mu$ is close to zero, we introduce the coordinate transformation $(U^\mathrm{c},\mu,d)=(\nu^{a_1} \tilde{U}^\mathrm{c}, \nu^{a_2} \tilde{\mu}, \nu^{a_3} \tilde{d})$, where $0\leq\nu\ll1$ and $(\tilde{U}^\mathrm{c}, \tilde{\mu}, \tilde{d})$ now lies on the unit sphere. This change of coordinates is invertible away from the origin $(U^\mathrm{c},\mu,d)=0$ or, alternatively, for $\nu>0$. The key is that for an appropriate choice of $a_1,a_2,a_3$ there is a constant $b\geq1$ so that
}
\[
\mathcal{F}^\mathrm{c}(\nu^{a_1} \tilde{U}^\mathrm{c}, \nu^{a_2} \tilde{\mu}, \nu^{a_3} \tilde{d}) = \nu^b \left(\mathcal{F}^\mathrm{c}(\tilde{U}^\mathrm{c}, \tilde{\mu}, \tilde{d}) + \mathcal{O}(\nu) \right), \quad
|(\tilde{U}^\mathrm{c}, \tilde{\mu}, \tilde{d})| = 1.
\]
\new{Hence, all zeros of $\mathcal{F}^\mathrm{c}(U^\mathrm{c},\mu,d)$ must lie near zeros of $\mathcal{F}^\mathrm{c}(\tilde{U}^\mathrm{c}, \tilde{\mu}, \tilde{d})$, and we can therefore focus on the desingularized equation $\mathcal{F}^\mathrm{c}(\tilde{U}^\mathrm{c}, \tilde{\mu}, \tilde{d})=0$ with arguments on the unit sphere: if the set of zeros of the desingularized equation consists entirely of regular zeros and generically unfolded bifurcations, then it persists robustly for $0<\nu\ll1$, and no additional roots can appear for $\nu>0$. In practice, we  will use  a directional blowup in our proofs below by setting $\tilde{d}=1$ or $\tilde{\mu}=1$ and will not consider the other directional blowups that together parametrize the entire sphere as they do not contribute additional solutions. We refer to \cite[\S7]{Kuehn} for references and additional details on geometric blowup.}

Throughout the proofs we will make use of the following. \new{Through two independent $\mu$-dependent coordinate transformations of the $u$ variable near $(u,\mu)=(0,0)$ and $(u_+(0),0)$, respectively}, we can bring the Taylor expansion of $f(u,\mu)$ at $(0,0)$ into the form
\begin{equation}\label{NormalForm1}
f(u,\mu) = -\mu u + u^3 + \mathcal{O}(\mu^2 u + \mu u^3 + u^5)
\end{equation}
\new{and achieve that $u_+(0)=1$. Using a similar change of coordinates near $(u,\mu)=(1,1)$}, we can bring the Taylor expansion of $f(u,\mu)$ at $(1,1)$ into the form
\begin{equation}\label{NormalForm2}
\new{f(u,\mu) = f(1+\check{u},1-\check{\mu}) = \check{\mu} - \check{u}^2 - b \check{\mu}\check{u} + \mathcal{O}(\check{\mu}^2 + \check{\mu} \check{u}^2 + \check{u}^3),}
\end{equation}
\new{for some constant $b$.} These changes of coordinates will greatly simplify the analysis in what follows.


\subsection{Nearest-neighbour coupling}\label{subsec:NNProof}


Throughout this section we will consider $N \geq 4$ and take $m = 1$, representing nearest-neighbour connections in the system \eqref{LDSI}. In \new{this subsection, we prove} Theorem~\ref{thm:Sparse} with $m = 1$, while the case of $m = 2$ is left to the following subsection. Per the discussion at the beginning of this section, we need only continue the \new{connections} in $\Gamma_\mathrm{sparse}\setminus\mathcal{E}$ near $\mu = 0,1$. We begin with the following lemma which continues the patterns near $\mu = 0$. 

\begin{lem}\label{lem:Pitchfork1} 
Fix $m=1$ and $2 \leq k \leq \lfloor \frac N 2 \rfloor + 1$, then the following is true for \eqref{SteadyLDS}. There are constants $d_1,\mu_1 > 0$ and a smooth function $\mu_l:[0,d_1] \to [0,\mu_1]$ such that for each $d \in (0,d_1]$\new{,} there is a pair of $\kappa$-symmetric solutions $U_l(\mu,d)$ and $V_l(\mu,d)$ of \eqref{SteadyLDS} that bifurcate at a fold bifurcation at $\mu = \mu_l(d)$ and exist for all $\mu \in [\mu_l(d),\mu_1]$. These solutions are smooth in $(\mu,d)$, and for each fixed $\mu$, we have $U_l(\mu,d) \to \bar{U}^{(k-1)}(\mu)$ and $V_l(\mu,d) \to \bar{V}^{(k)}(\mu)$ as $d \to 0^+$. The function $\mu_l(d)$ satisfies $\mu_l(d) = \frac{3}{\sqrt[3]{2}}d^\frac{2}{3} + \mathcal{O}(d)$.
\end{lem}

\begin{proof}
	We will fix $k \in \{2,\dots,\lfloor \frac N 2 \rfloor + 1\}$ and construct symmetric solutions of \eqref{SteadyLDS} near the pattern
\[
\bar{U}^{(k)}(0)=\bar{V}^{(k)}(0) = \left\{\begin{array}{ll}
1 & n < k \\
0        & \mathrm{otherwise}
\end{array}\right.
\]
for $(\mu,d)$ near zero. We reduce patterns to the index set $I$, defined in \eqref{IndexSet}, using the aforementioned $\kappa$ symmetry. To solve $\mathcal{F}(U,\mu,d)=0$, we note that $\mathcal{F}(\bar{U}^{(k)}(0),0,0)=0$ and that the linearization of $\mathcal{F}$ is given by
\[
(\mathcal{F}_U(\bar{U}^{(k)}(0),0,0)v)_n = \left\{ \begin{array}{cl}
f_u(1,0) v_n & n < k \\
0                & n \geq k.
\end{array}\right.
\]
Writing $U^+:=U|_{n < k}$ and $U^c:=U|_{n \geq k}$, and using that $f_u(1,0)\neq0$, we can apply the implicit function theorem to conclude that $\mathcal{F}(U,\mu,d)=0$ restricted to the index set $1 \leq n < k$ has a unique solution $U^+(u^c,\mu,d) \in\R^{k-1}$ for each $U^c\in\R^{\lfloor \frac N 2 \rfloor -k+1}$ and $(\mu,d)$ near zero. Furthermore, this solution depends smoothly on its arguments, and so we have 
\begin{equation}\label{PitchSol}
U^+(U^c,\mu,d) = 1 + \mathcal{O}(|\mu|+|d| \|U^c\|).
\end{equation}
To solve \eqref{SteadyLDS} for the indices $n \geq k$, we introduce the scaling
\begin{equation}\label{e:scaling}
\mu = \nu^2, \qquad
d = \nu^3 \tilde{d}, \qquad
u_n = \nu^{n - k + 1} \tilde{u}_n, \qquad
\end{equation}
for $n\geq k$ with $|\nu|\ll1$. Substituting these expressions into \eqref{SteadyLDS} for each $n \geq k$, we see that \eqref{SteadyLDS} restricted to the index set $I\setminus I_+$ becomes
\[
\begin{array}{lcl}
n = k:					&& 0 = \nu^3(\tilde{d} - \tilde{u}_{k} + \tilde{u}^3_{k}) + \mathcal{O}(\nu^4) \\
n > k:						&& 0 = \nu^{n - k + 3}(\tilde{d}\tilde{u}_{n-1} - \tilde{u}_n) + \mathcal{O}(\nu^{n-k+4}),
\end{array}
\]
where we recall that we have introduced a change of variable to bring the system to the normal form \eqref{NormalForm1}. Upon dividing by the leading factors in $\nu$, we arrive at the system
\begin{equation}\label{PitchSol2}
\begin{array}{lcl}
n = k:					&& 0 = \tilde{d} - \tilde{u}_{k} + \tilde{u}^3_{k} + \mathcal{O}(\nu) \\
n > k:						&& 0 = \tilde{d}\tilde{u}_{n-1} - \tilde{u}_n + \mathcal{O}(\nu) \\
\end{array}
\end{equation}
for which we can see that at the index $n = k$ a fold bifurcation takes place at $(\tilde u_k,\tilde{d},\nu) = (\frac{1}{\sqrt{3}},\frac{2}{3\sqrt{3}},0)$. \new{Extending this fold bifurcation to the full system for the indices $n \geq k$ into $\nu > 0$ now follows as in the proof of \cite[Lemma~3.2]{Bramburger2}. As this method underpins many of the proofs that follow, we will include the details to complete this proof.} 

\new{First, by setting $\nu = 0$ we can parametrize the fold bifurcation at index $n = k$ by}
\[
    (\tilde u_k,\tilde d)(s) = (s,s(1-s^2)), \quad s \in [0,1].
\]
\new{Notice that this parametrization connects $(0,0)$ to $(1,0)$ as $s$ increases through the interval $[0,1]$. Furthermore, the fold bifurcation takes place when $s = \frac{1}{\sqrt{3}}$, giving $\tilde{d}_\mathrm{sn} = \frac{2}{3\sqrt{3}}$. Continuing with $\nu = 0$, the remaining $n > k$ indices in \eqref{PitchSol2} can be written as the matrix equation}
\[
    (-\mathbb{I} + \tilde d(s)Z)\tilde U_{n > k} = \tilde d(s)\tilde u_k(s)\hat{e}_1
\]
\new{where $\mathbb{I}$ denotes the identity matrix, $Z$ is a matrix with ones along the diagonal above the main diagonal and zeros elsewhere, $\tilde U_{n > k}$ is the vector of $\tilde u_n$ with $n > k$, and $\hat{e}_1$ is the vector with 1 in its first entry and zeros elsewhere. Since $\|Z\| \leq 1$ and $\tilde{d}(s) \in [0,\frac{2}{3\sqrt{3}}]$ for all $s\in[0,1]$, it follows that the matrix $(-1 + \tilde d(s)Z)$ is invertible for all $s\in[0,1]$. Therefore, for each $s\in[0,1]$ we can uniquely solve for $\tilde{U}_{n>k}$ in terms of $(\tilde u_k,\tilde d)(s)$ which parametrize the fold bifurcation. This extends the fold bifurcation into the indices $n > k$ for $\nu = 0$.}

\new{The persistence of the fold bifurcation at $\nu = 0$ into $0 < \nu \ll 1$ can first be obtained by writing the right-hand-side of \eqref{PitchSol2} compactly as $G(\tilde U^c,\tilde d,\nu)$. Note that the branch constructed above, denoted $(\tilde U^c,\tilde d)(s)$ satisfies $G(\tilde U^c(s),\tilde d(s),0) = 0$ and the derivative $G_{(\tilde U^c,\tilde d)}(\tilde U^c(s),\tilde d(s),0)$ has full rank for all $s\in[0,1]$. We can therefore apply the implicit function theorem and use persistence results for fold bifurcations to conclude that the branch persists for sufficiently small $\nu > 0$. Moreover, the unique fold bifurcation takes place at $\tilde{d} = \tilde d_\mathrm{sn}(\nu)$ with $\tilde d_\mathrm{sn}(0) = \frac{2}{3\sqrt{3}}$. This completes the proof of the lemma.}
\end{proof} 

We now turn to the continuation when $\mu = 1$. Recall that from \eqref{NormalForm2} we have that upon changing coordinates we can bring the Taylor expansion of $f(u,\mu)$ about $(1,1)$ into the form
\[
\new{f(u,\mu) = f(1+\check{u},1-\check{\mu}) = \check{\mu} - \check{u}^2 - b \check{\mu}\check{u} + \mathcal{O}(\check{\mu}^2 + \check{\mu} \check{u}^2 + \check{u}^4).}
\]
This leads to the following result which continues the connections in $\Gamma$ near $\mu = 1$ and completes the proof of Theorem~\ref{thm:Sparse} when $m = 1$.

\begin{lem}\label{lem:Saddle1} 
Fix $m=1$ and $1 \leq k \leq \lfloor \frac N 2 \rfloor$, then the following is true for \eqref{SteadyLDS}. There exist constants $d_2,\mu_2 > 0$ and a smooth function $\mu_r:[0,d_2] \to [\mu_2,1]$ such that for each fixed $d \in (0,d_2]$\new{,} there is a pair of $\kappa$-symmetric solutions $U_r(\mu,d)$ and $V_r(\mu,d)$ of \eqref{SteadyLDS} that bifurcate at a fold bifurcation at $\mu = \mu_r(d)$ and exist for all $\mu \in [\mu_2,\mu_r(d)]$. These solutions are smooth in $(\mu,d)$, and for each fixed $\mu$\new{,} we have $U_r(\mu,d) \to \bar{U}^{(k)}(\mu)$ and $V_r(\mu,d) \to \bar{V}^{(k)}(\mu)$ as $d \to 0^+$. The function $\mu_r(d)$ is given by $\mu_r(d)=1-d+\mathcal{O}(d^\frac{3}{2})$.
\end{lem}

\begin{proof}
	We again restrict to the index set $I$ and extend the solutions by symmetry to $n \geq \lfloor N/2 \rfloor + 2$. As the branch passes near $\mu = 1$, the cell $u_k$ changes from $u_-(\mu)$ to $u_+(\mu)$, while the remaining cells stay near $0$ or $u_+(\mu)$. We have that $\mathcal{F}(\bar{U}^{(k)}(1),1,0)=0$ and that the linearization of $\mathcal{F}$ about this solution is given by 
\[
(\mathcal{F}_U(\bar{U}^{(k)}(1),1,0)v)_{n} =
\left\{ \begin{array}{cl}
f_u(0,1) v_{n} & n > k \\
0                & n \leq k.
\end{array}\right.
\]
Writing $U^0:=U|_{n > k}$ and $U^c:=U|_{n \leq k}$, and using that $f_u(0,1)\neq0$, we can apply the implicit function theorem to find that $\mathcal{F}(U,\mu,d) = 0$ restricted to the index set $n > k$ has a unique solution $U^0(U^c,\mu,d) \in \ell^\infty(I|_{n > k})$ for each $U^c \in\R^k$ and $(\mu,d)$ near $(1,0)$. This solution depends smoothly on its arguments, and in particular, has the expansion
\begin{equation}\label{Saddle1}
U^0(U^c,\mu,d) = \mathcal{O}(|\mu - 1| + |d| |U^c|).
\end{equation}
To solve \eqref{SteadyLDS} on the index set $n\leq k$, we introduce the scaling
\[
\mu = 1 -\nu^2, \quad d = \nu^2\tilde{d}, \quad u_n = 1 + \nu\tilde{u}_n,
\]
where $1 \leq n \leq k$ and $|\nu| \ll 1$. Expanding $\mathcal{F}(U,\mu,d)=0$ restricted to the index set $n \in\{1,\dots,k\}$ in powers of $\nu$ and dividing by the leading factor in $\nu$, we arrive at the finite system
\begin{equation}\label{Saddle2}
\begin{array}{lcl}
n = k: && 0 = -\tilde{d} + 1 - \tilde{u}_k^2 + \mathcal{O}(\nu) \\
n < k: && 0 = 1 - \tilde{u}_n^2 + \mathcal{O}(\nu) \\
\end{array}
\end{equation}
where we used \eqref{Saddle1} to simplify the equation with $n = k$ using the connection to the element at index $n = k+1$. We now see that at the index $n = k$, a fold bifurcation takes place at $(\tilde u_k,\tilde d,\nu) = (0,1,0)$. Extending this fold bifurcation to the full system on the indices $n \geq k$ and into $\nu > 0$ follows as in \new{the previous lemma} and is omitted. This completes the proof of the lemma.
\end{proof} 


\subsection{Next-nearest-neighbour coupling}\label{subsec:NNNProof}


In this subsection we provide analogous results to the nearest-neighbour bifurcation branches, but with $m = 2$, representing next-nearest-neighbour connections. The results of this subsection therefore complete the proof of Theorem~\ref{thm:Sparse} with $m = 2$. We will consider $N \geq 7$, since $N = 4,5$ and $m = 2$ represent all-to-all connections, and $N = 6$ with $m = 2$ has an extra symmetry that can be exploited, as presented in Theorem~\ref{thm:6,2}. We present the following result, analogous to Lemma~\ref{lem:Pitchfork1} above.

\begin{lem}\label{lem:Pitchfork2} 
Fix $m= 2$ and $3 \leq k \leq \lfloor \frac N 2 \rfloor + 1$, then the following is true for \eqref{SteadyLDS}. There are constants $d_1,\mu_1 > 0$ and a smooth function $\mu_l:[0,d_1] \to [0,\mu_1]$ such that for each $d \in (0,d_1]$\new{,} there is a pair of $\kappa$-symmetric solutions $U_l(\mu,d)$ and $V_l(\mu,d)$ of \eqref{SteadyLDS} that bifurcate at a fold bifurcation at $\mu = \mu_l(d)$ and exist for all $\mu \in [\mu_l(d),\mu_1]$. These solutions are smooth in $(\mu,d)$, and for each fixed $\mu$, we have $U_l(\mu,d) \to \bar{U}^{(k-1)}(\mu)$ and $V_l(\mu,d) \to \bar{V}^{(k)}(\mu)$ as $d \to 0^+$. The function $\mu_l(d)$ satisfies $\mu_l(d) = 3d^\frac{2}{3} + \mathcal{O}(d)$.
\end{lem}

\begin{proof}
	Fix $m=2$ and $3 \leq k \leq \lfloor \frac N 2 \rfloor + 1$. Then, following as in Lemma~\ref{lem:Pitchfork1} to continue the elements with $n < k$ into $d > 0$ near $\mu = 0$ using the implicit function theorem. We again get the asymptotic expansion \eqref{PitchSol}, and for the indices with $n \geq k$ we introduce the scaling
\begin{equation}\label{e:scaling2}
\mu = \nu^2, \qquad
d = \nu^3 \tilde{d}, \qquad
u_n = \nu^{\lfloor\frac{n - k }{2}\rfloor+ 1} \tilde{u}_n.
\end{equation}
That is, $u_k = \nu \tilde{u}_k$, $u_{k+1} = \nu \tilde{u}_{k+1}$, $u_{k+2} = \nu^2 \tilde{u}_{k+2}$, $u_{k+3} = \nu^2 \tilde{u}_{k+3}$, and so on. Putting these rescaled variables into \eqref{SteadyLDS}, expanding in powers of $\nu$, and dividing off the leading power in $\nu$ brings us to the equations, analogous to \eqref{PitchSol2}, 
	\begin{equation}\label{PitchSol3}
	\begin{array}{lcl}
	n = k:					&& 0 = 2\tilde{d} - \tilde{u}_{k} + \tilde{u}^3_{k} + \mathcal{O}(\nu) \\
	n = k+1:					&& 0 = \tilde{d} - \tilde{u}_{k+1} + \tilde{u}^3_{k+1} + \mathcal{O}(\nu) \\
	n \geq k+2:				&& 0 = \tilde{d}\tilde{u}_{n-2} + \tilde{d}\tilde{u}_{n-1} - \tilde{u}_n + \mathcal{O}(\nu) \\
	\end{array}
	\end{equation}
	where we have used the fact that at index $n = k$ we are connected to the elements at index $k-1,k-2$, both of which are equal to $1$ at $(\mu,d) = (0,0)$. Similarly, since $m =2$, the element at index $n = k+1$ is only connected to one element (at $n = k-1$) that is equal to $1$ at $(\mu,d) = (0,0)$. The proof is now the same as that of Lemma~\ref{lem:Pitchfork1}. 
\end{proof}

Let us now finish the case of bifurcations near $\mu = 0$ by considering the case $k = 2$. We present the following result.

\begin{lem}\label{lem:Pitchfork2_2} 
Fix $m= k = 2$, then the following is true for \eqref{SteadyLDS}. There are constants $d_1,\mu_1 > 0$ and a smooth function $\mu_l:[0,d_1] \to [0,\mu_1]$ such that for each $d \in (0,d_1]$\new{,} there is a pair of $\kappa$-symmetric solutions $U_l(\mu,d)$ and $V_l(\mu,d)$ of \eqref{SteadyLDS} that bifurcate at a fold bifurcation at $\mu = \mu_l(d)$ and exist for all $\mu \in [\mu_l(d),\mu_1]$. These solutions are smooth in $(\mu,d)$, and for each fixed $\mu$\new{,} we have $U_l(\mu,d) \to \bar{U}^{(1)}(\mu)$ and $V_l(\mu,d) \to \bar{V}^{(2)}(\mu)$ as $d \to 0^+$. The function $\mu_l(d)$ satisfies $\mu_l(d) = \frac{3}{\sqrt[3]{2}}d^\frac{2}{3} + \mathcal{O}(d)$.
\end{lem}

\begin{proof}
	Following as in the proof of Lemma~\ref{lem:Pitchfork2} up to \eqref{PitchSol3}, we now arrive at the equations  
	\begin{equation}\label{PitchSol4}
	\begin{array}{lcl}
	n = 2:					&& 0 = \tilde{d} - \tilde{u}_{2} + \tilde{u}^3_{2} + \nu(\tilde{u}_3-3\tilde{u}_2) + \mathcal{O}(\nu^2) \\
	n = 3:					&& 0 = \tilde{d} - \tilde{u}_{3} + \tilde{u}^3_{3} + \nu(\tilde{u}_2-4\tilde{u}_3) + \mathcal{O}(\nu^2) \\
	n \geq 4:				&& 0 = \tilde{d}\tilde{u}_{n-2} + \tilde{d}\tilde{u}_{n-1} - \tilde{u}_n + \mathcal{O}(\nu) \\
	\end{array}
	\end{equation}
	where now, since $m = 2$, we have that the elements at indices $n = 2,3$ are both connected to the single element at $n=1$ that equals 1 at $(\mu,d) = (0,0)$. Note that the equations at indices $n = 2,3$ agree at the lowest order in $\nu$, but differ at $\mathcal{O}(\nu)$. This comes from the symmetry imposed by our assumption that the solution is invariant with respect to the action of $\kappa$, which enforces that $u_N = u_2$ and thus eliminating one of the connections at $n = 2$. In contrast, we have exactly four self-interactions at index $n = 3$ in \eqref{PitchSol4} at order $\nu$ since the element at index $n = 3$ has no neighbours that have a symmetric restriction imposed \new{on} them. From here the proof now follows as in \cite[Lemma~3.3]{Bramburger2}, \new{but we include the details to keep this manuscript self-contained}.
	
	\new{Let us introduce the new variables $(d_0,v_2,v_3)$, given by}
	\[
	    \tilde{d} = \frac{2}{3\sqrt{3}} + \nu d_0, \quad \tilde{u}_2 = \frac{1}{\sqrt{3}} + \nu^\frac{1}{2}v_2, \quad \quad \tilde{u}_3 = \frac{1}{\sqrt{3}} + \nu^\frac{1}{2}v_3
	\]
	\new{so that \eqref{PitchSol4} becomes}
	\[
	\begin{array}{lcl}
	n = 2:					&& 0 = d_0 + \sqrt{3}v_2^2 - \frac{4}{27} + \mathcal{O}(\nu^\frac{1}{2}) \\
	n = 3:					&& 0 = d_0 + \sqrt{3}v_3^2 - \frac{2}{9} + \mathcal{O}(\nu^\frac{1}{2}) \\
	n = 4:                  && 0 = \frac{4}{9} - \tilde u_4 + \mathcal{O}(\nu^\frac{1}{2}) \\
	n = 5:                  && 0 = \frac{2}{9} + \frac{2}{3\sqrt{3}}\tilde u_4 - \tilde u_5 + \mathcal{O}(\nu^\frac{1}{2}) \\ 
	n \geq 6:				&& 0 = \frac{2}{3\sqrt{3}}\tilde{u}_{n-2} + \frac{2}{3\sqrt{3}}\tilde{u}_{n-1} - \tilde{u}_n + \mathcal{O}(\nu) \\
	\end{array}
	\]
	\new{after dividing off the leading factor of $\nu$ in the equations for $n = 2,3$. The equations for indices $n \geq 4$ can be solved uniquely for all bounded $(d_0,v_2,v_3)$ and sufficiently small $\nu > 0$, and therefore we focus exclusively on the equations for the indices $n = 2,3$. Setting $\nu = 0$ we find that a fold bifurcation takes place in the equation at index $n = 2$ at $(d_0,v_2) = (\frac{4}{27},0)$. Furthermore, we can solve the equation at $n = 3$ uniquely for $v_3$ for $(d_0,v_2,\nu)$ in a neighbourhood of this fold bifurcation, leaving only the equation at $n = 2$. Then, as in the proof of Lemma~\ref{lem:Pitchfork1}, the fold bifurcation at $n = 2$ with $\nu = 0$ can be shown to persist into $0 < \nu \ll 1$, thus concluding the proof of the lemma.}
\end{proof} 

We now turn to the continuation when $\mu = 1$, which will similarly follow the proof of Lemma~\ref{lem:Saddle1}.

\begin{lem}\label{lem:Saddle2} 
Fix $m=2$ and $1 \leq k \leq \lfloor \frac N 2 \rfloor$, then the following is true for \eqref{SteadyLDS}. There exist constants $d_2,\mu_2 > 0$ and a smooth function $\mu_r:[0,d_2] \to [\mu_2,1]$ such that each fixed $d \in (0,d_2]$\new{,} there is a pair of $\kappa$-symmetric solutions $U_r(\mu,d)$ and $V_r(\mu,d)$ of \eqref{SteadyLDS} that bifurcate at a fold bifurcation at $\mu = \mu_r(d)$ and exist for all $\mu \in [\mu_2,\mu_r(d)]$. These solutions are smooth in $(\mu,d)$, and for each fixed $\mu$\new{,} we have $U_r(\mu,d) \to \bar{U}^{(k)}(\mu)$ and $V_r(\mu,d) \to \bar{V}^{(k)}(\mu)$ as $d \to 0^+$. The function $\mu_r(d)$ is given by $\mu_r(d)=1-2d+\mathcal{O}(d^\frac{3}{2})$.
\end{lem}

\begin{proof}
	This proof is almost the same as that of Lemma~\ref{lem:Saddle1} with \eqref{Saddle2} replaced by 
	\begin{equation}\label{Saddle3}
	\begin{array}{lcl}
	n = k: && 0 = -2\tilde{d} + 1 - \tilde{u}_k^2 + \mathcal{O}(\nu) \\
	n = k-1: && 0 = -\tilde{d} + 1 - \tilde{u}_{k-1}^2 + \mathcal{O}(\nu) \\
	n \leq k-2: && 0 = 1 - \tilde{u}_n^2 + \mathcal{O}(\nu). \\
	\end{array}
	\end{equation}
	The subsequent analysis proceeds in the same way as Lemma~\ref{lem:Saddle1}.
\end{proof} 


\subsection{Almost all-to-all coupling}\label{subsec:AlmostProof}


Let us begin with the proof of Theorem~\ref{thm:6,2}, which features $(N,m) = (6,2)$. With the exception of the connections involving curves continued from $\bar W^{(23)}(\mu)$, the analysis is the same as in the proof of Theorem~\ref{thm:Sparse}. Hence, we will focus exclusively on the connections involving $\bar W^{(23)}(\mu)$. There is only one such connection near $\mu = \new{0}$, which is illustrated for small $d > 0$ in Figure~\ref{fig:Almost_All2All}(i). We prove this with the following lemma. 

\begin{figure} 
\center
\includegraphics[width = 0.9\textwidth]{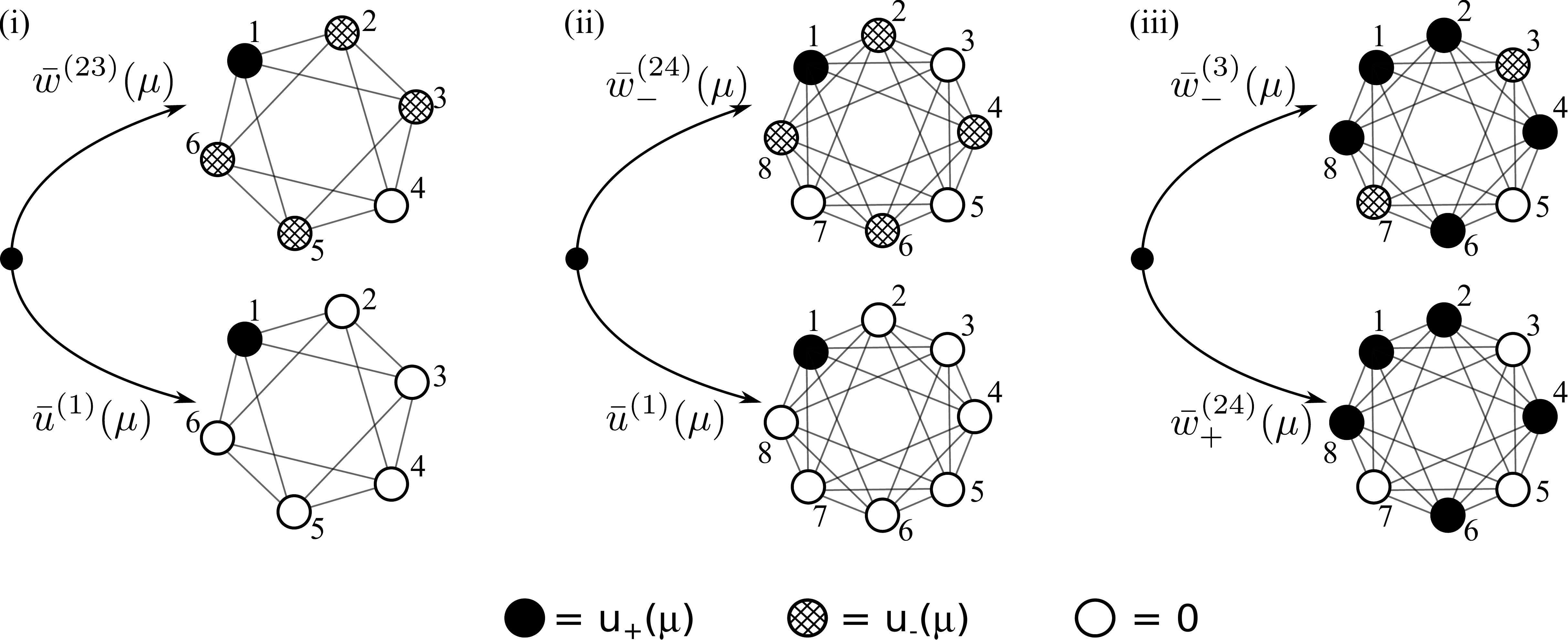} 
\caption{Atypical bifurcations near $\mu = 0$ for $N = 6,8$ and $m = \lfloor N/2 \rfloor - 1$. Panel (i) contains an illustration of the proof of Lemma~\ref{lem:PitchforkN=6}. Panels (ii) and (iii) illustrate the bifurcations near $\mu = 0$ when $N = 8$, as presented in Lemmas~\ref{lem:PitchforkN=8_1} and \ref{lem:PitchforkN=8_2}. In all images black circles represent elements continued from $u_+(\mu)$ into $d > 0$, shaded grey circles are continued from $u_-(\mu)$, and white circles are continued from $0$. Connections are represented by lines connecting the circles.}
\label{fig:Almost_All2All}
\end{figure} 

 \begin{lem}\label{lem:PitchforkN=6} 
Fix $(N,m) = (6,2)$, then the following is true for \eqref{SteadyLDS}. There are constants $d_1,\mu_1 > 0$ and a smooth function $\mu_l:[0,d_1] \to [0,\mu_1]$ such that for each $d \in (0,d_1]$, there is a pair of $\kappa$-symmetric solutions $U_l(\mu,d)$ and $V_l(\mu,d)$ of \eqref{SteadyLDS} that bifurcate at a fold bifurcation at $\mu = \mu_l(d)$ and exist for all $\mu \in [\mu_l(d),\mu_1]$. These solutions are smooth in $(\mu,d)$, and for each fixed $\mu$\new{,} we have $U_l(\mu,d) \to \bar{U}^{(1)}(\mu)$ and $V_l(\mu,d) \to \bar{W}^{(23)}(\mu)$ as $d \to 0^+$. The function $\mu_l(d)$ satisfies $\mu_l(d) = \frac{3}{\sqrt[3]{2}}d^\frac{2}{3} + \mathcal{O}(d)$.
\end{lem}  

\begin{proof}
	This proof is handled in exactly the same way as the other proofs for bifurcations near $\mu = 0$, and so we will only seek to highlight why we have a different connection when $N = 6$ and $m = 2$. Using the implicit function theorem, we can solve \eqref{SteadyLDS} at the index $n=1$ in a neighbourhood of $(d,\mu) = (0,0)$ to get 
	\[
		u_1 = u_1(U^c,\mu,d) = 1 + \mathcal{O}(|\mu|+|d| |U^c|),
	\]
	where $U^c = (u_2,u_3,u_4)$ since we have imposed $u_5 = u_3$ and $u_6 = u_2$ by symmetry. It then remains to solve \eqref{SteadyLDS} for the indices associated with $U^c$, i.e. $n = 2,3,4$. Specifically, we are required to solve
	\begin{equation}
		\begin{split}
			d(u_1 + u_3 + u_4 - 3u_2) + f(u_2,\mu) &= 0, \\
			d(u_1 + u_2 + u_4 - 3u_3) + f(u_3,\mu) &= 0, \\
			2d(u_2 + u_3 - 2u_4) + f(u_4,\mu) &= 0,
		\end{split}
	\end{equation}
	where we have simplified the equations using the symmetric restrictions $u_5 = u_3$ and $u_6 = u_2$. One can see that the resulting equations are invariant with respect to the action $(u_2,u_3) \mapsto (u_3,u_2)$, and therefore we may restrict ourselves to the symmetric subspace that has $u_2 = u_3$. Upon imposing this restriction we may then \new{proceed} as in the previous proofs to obtain the desired result.   
\end{proof} 

Near $\mu = 1$ we find that the branch continued from $\bar W^{(23)}(\mu)$ connects to $\bar{U}^{(3)}(\mu)$, summarized in the following lemma and illustrated in Figure~\ref{fig:Almost_All2All_2}(i). The lemma is stated without proof since it is identical to the previous lemmas. 

\begin{figure} 
\center
\includegraphics[width = 0.9\textwidth]{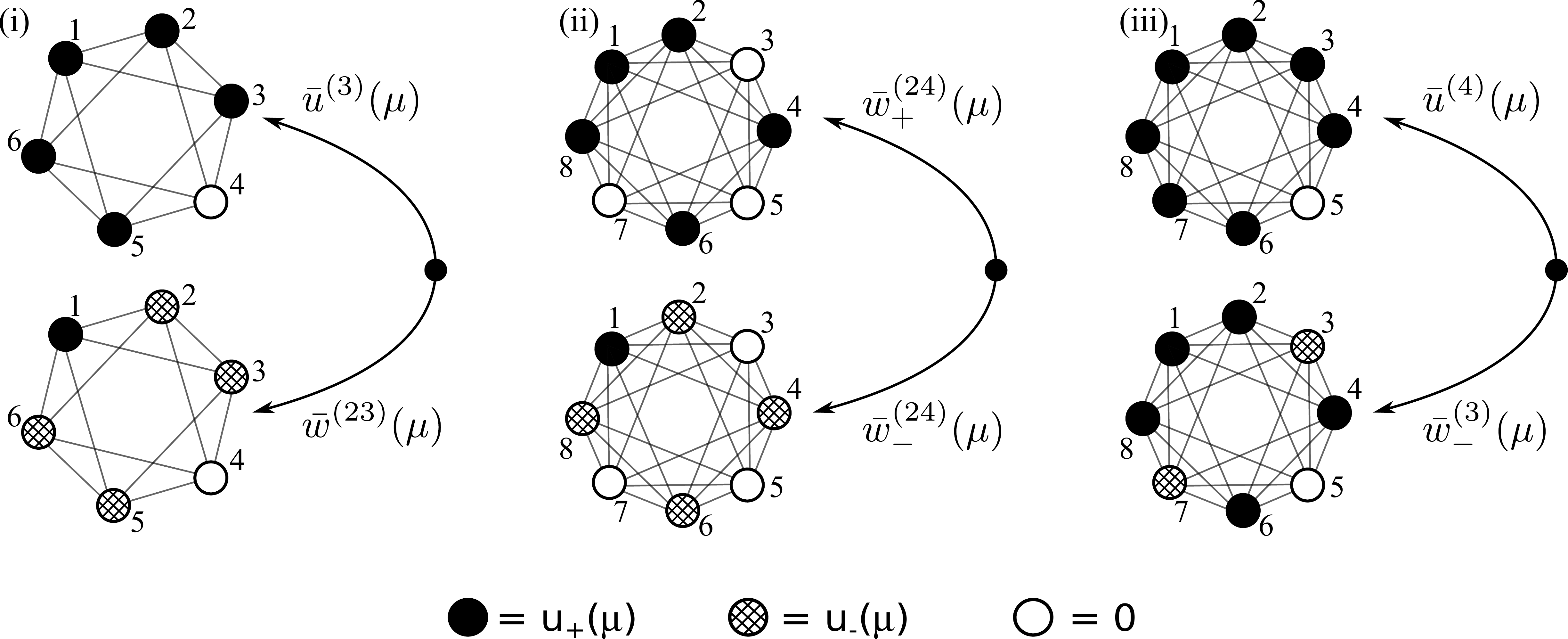} 
\caption{Atypical bifurcations near $\mu = 1$ for $N = 6,8$ and $m = \lfloor N/2 \rfloor - 1$. Panel (i) contains an illustration of the proof of Lemma~\ref{lem:SaddleN=6}. Panels (ii) and (iii) illustrate the bifurcations near $\mu = 1$ when $N = 8$, as presented in Lemmas~\ref{lem:SaddleN=8_1} and \ref{lem:SaddleN=8_2}. Shading and connections are the same as in Figure~\ref{fig:Almost_All2All}.}
\label{fig:Almost_All2All_2}
\end{figure} 

\begin{lem}\label{lem:SaddleN=6} 
Fix $(N,m) = (6,2)$, then the following is true for \eqref{SteadyLDS}. There exist constants $d_2,\mu_2 > 0$ and a smooth function $\mu_r:[0,d_2] \to [\mu_2,1]$ such that for each fixed $d \in (0,d_2]$, there is a pair of $\kappa$-symmetric solutions $U_r(\mu,d)$ and $V_r(\mu,d)$ of \eqref{SteadyLDS} that bifurcate at a fold bifurcation at $\mu = \mu_r(d)$ and exist for all $\mu \in [\mu_2,\mu_r(d)]$. These solutions are smooth in $(\mu,d)$, and for each fixed $\mu$\new{,} we have $U_r(\mu,d) \to \bar{U}^{(3)}(\mu)$ and $V_r(\mu,d) \to \bar{W}^{(23)}(\mu)$ as $d \to 0^+$. The function $\mu_r(d)$ is given by $\mu_r(d)=1-2d+\mathcal{O}(d^\frac{3}{2})$.
\end{lem}

Turning now to the case of $(N,m) = (8,3)$, we remark that much of the analysis is similar. Recall that the set $\Gamma_{8,3}$ is characterized by following along solution branches at $d = 0$ given by the following sequence: 
\[
	\bar V^{(1)}(\mu) \to \bar{U}^{(1)}(\mu) \to \bar{W}^{(24)}_-(\mu) \to \bar{W}^{(24)}_+(\mu) \to \bar{W}^{(3)}_-(\mu) \to \bar{U}^{(4)}(\mu) \to \bar{V}^{(4)}(\mu). 	
\] 
The connections between $\bar{U}$ and $\bar{V}$ elements are handled as in the proof of Theorem~\ref{thm:Sparse}, while the connections between any of the $\bar{W}$ elements are summarized in the following lemmas. We refer the reader to Figure~\ref{fig:Almost_All2All} for illustrations of the connections near $\mu = 0$ and to Figure~\ref{fig:Almost_All2All_2} for illustrations near $\mu = 1$. The lemmas are listed according to the order of the sequence above, while the proofs are omitted since they are similar to much of the work performed in this section. Together these results complete the proof of Theorem~\ref{thm:8,3}.

\begin{lem}\label{lem:PitchforkN=8_1} 
Fix $(N,m) = (8,3)$, then the following is true for \eqref{SteadyLDS}. There are constants $d_1,\mu_1 > 0$ and a smooth function $\mu_l:[0,d_1] \to [0,\mu_1]$ such that for each $d \in (0,d_1]$\new{,} there is a pair of $\kappa$-symmetric solutions $U_l(\mu,d)$ and $V_l(\mu,d)$ of \eqref{SteadyLDS} that bifurcate at a fold bifurcation at $\mu = \mu_l(d)$ and exist for all $\mu \in [\mu_l(d),\mu_1]$. These solutions are smooth in $(\mu,d)$, and for each fixed $\mu$\new{,} we have $U_l(\mu,d) \to \bar{U}^{(1)}(\mu)$ and $V_l(\mu,d) \to \bar{W}^{(24)}_-(\mu)$ as $d \to 0^+$. The function $\mu_l(d)$ satisfies $\mu_l(d) = \frac{3}{\sqrt[3]{2}}d^\frac{2}{3} + \mathcal{O}(d)$.
\end{lem}  

\begin{lem}\label{lem:SaddleN=8_1} 
Fix $(N,m) = (8,3)$, then the following is true for \eqref{SteadyLDS}. There exist constants $d_2,\mu_2 > 0$ and a smooth function $\mu_r:[0,d_2] \to [\mu_2,1]$ such that for each fixed $d \in (0,d_2]$\new{,} there is a pair of $\kappa$-symmetric solutions $U_r(\mu,d)$ and $V_r(\mu,d)$ of \eqref{SteadyLDS} that bifurcate at a fold bifurcation at $\mu = \mu_r(d)$ and exist for all $\mu \in [\mu_2,\mu_r(d)]$. These solutions are smooth in $(\mu,d)$, and for each fixed $\mu$\new{,} we have $U_r(\mu,d) \to \bar{W}^{(24)}_-(\mu)$ and $V_r(\mu,d) \to \bar{W}^{(24)}_+(\mu)$ as $d \to 0^+$. The function $\mu_r(d)$ is given by $\mu_r(d)=1-2d+\mathcal{O}(d^\frac{3}{2})$.
\end{lem}

\begin{lem}\label{lem:PitchforkN=8_2} 
Fix $(N,m) = (8,3)$, then the following is true for \eqref{SteadyLDS}. There are constants $d_1,\mu_1 > 0$ and a smooth function $\mu_l:[0,d_1] \to [0,\mu_1]$ such that for each $d \in (0,d_1]$\new{,} there is a pair of $\kappa$-symmetric solutions $U_l(\mu,d)$ and $V_l(\mu,d)$ of \eqref{SteadyLDS} that bifurcate at a fold bifurcation at $\mu = \mu_l(d)$ and exist for all $\mu \in [\mu_l(d),\mu_1]$. These solutions are smooth in $(\mu,d)$, and for each fixed $\mu$\new{,} we have $U_l(\mu,d) \to \bar{W}^{(23)}_+(\mu)$ and $V_l(\mu,d) \to \bar{W}^{(3)}_-(\mu)$ as $d \to 0^+$. The function $\mu_l(d)$ satisfies $\mu_l(d) = \frac{3}{\sqrt[3]{2}}d^\frac{2}{3} + \mathcal{O}(d)$.
\end{lem}  

\begin{lem}\label{lem:SaddleN=8_2} 
Fix $(N,m) = (8,3)$, then the following is true for \eqref{SteadyLDS}. There exist constants $d_2,\mu_2 > 0$ and a smooth function $\mu_r:[0,d_2] \to [\mu_2,1]$ such that for each fixed $d \in (0,d_2]$\new{,} there is a pair of $\kappa$-symmetric solutions $U_r(\mu,d)$ and $V_r(\mu,d)$ of \eqref{SteadyLDS} that bifurcate at a fold bifurcation at $\mu = \mu_r(d)$ and exist for all $\mu \in [\mu_2,\mu_r(d)]$. These solutions are smooth in $(\mu,d)$, and for each fixed $\mu$\new{,} we have $U_r(\mu,d) \to \bar{W}^{(3)}_-(\mu)$ and $V_r(\mu,d) \to \bar{U}^{(4)}(\mu)$ as $d \to 0^+$. The function $\mu_r(d)$ is given by $\mu_r(d)=1-2d+\mathcal{O}(d^\frac{3}{2})$.
\end{lem}


\subsection{All-to-all coupling}\label{subsec:AllProof}


Let us now consider the case of all-to-all coupling. That is, $m = \lfloor \frac N 2 \rfloor$, and so every element is connected to all other elements. Then, our interest in \new{$S_k\times S_{N-k}$}-invariant solutions means that the system \eqref{AllLDS} reduces to solving the two equations
\begin{subequations} \label{all-to-all}
  \begin{empheq}[left=\empheqlbrace]{align}
    d(N-k)(v_2-v_1)+f(v_1,\mu) = 0, \label{all-to-all_1}\\
    dk(v_1-v_2)+f(v_2,\mu) = 0 \label{all-to-all_2}
  \end{empheq}
\end{subequations}
where $v_1$ denotes the values of the vector in $\R^N$ at the first $k$ indices and $v_2$ the values of the last $N-k$. As stated at the beginning of this section, we clearly have that for $\mu$ belonging to any compact interval of $[0,1]$ we can continue a solution of \eqref{Function} at $d = 0$ with $u_n \in \{0,u_\pm(\mu)\}$ for all $n = 2,\dots,N$ regularly into $d > 0$ with the implicit function theorem. Moreover, from the form of \eqref{all-to-all}, we may restrict $\mathcal{F}$ to a \new{$S_k\times S_{N-k}$}-invariant subspace to guarantee that solutions which are \new{$S_k\times S_{N-k}$}-invariant at $d = 0$ persist into $d > 0$ with the same symmetry. Hence, to prove Theorem~\ref{thm:All2All} we again need only check continuation into $d > 0$ of the connections between the elements $\bar A_\pm^{(k)}(\mu)$, $\bar B^{(k)}(\mu)$, $\bar{C}_\pm^{(k)}(\mu)$, and $\bar D^{(k)}(\mu)$ near $\mu = 0,1$ \new{(see Figure~\ref{fig:All2AllABCD} for visual demonstration)}.  

We begin with the following lemma which details the continuation into $d > 0$ of the connections between $\bar A_+^{(k)}(\mu)$ and $\bar B^{(k)}(\mu)$\new{,} and $\bar C_+^{(k)}(\mu)$ and $\bar D^{(k)}(\mu)$, near $\mu = 0$. 

\begin{lem}\label{lem:all_1} 
For any $N \geq 2$ and $k = 1,\dots, \new{\lfloor N/2 \rfloor}$, the following is true for \eqref{SteadyLDS}. 
\begin{compactenum}
\item There are constants $d_{ab},\mu_{ab} > 0$ and a smooth function $\mu_l:[0,d_{ab}] \to [0,\mu_{ab}]$ such that for each $d \in (0,d_{ab}]$\new{,} there is a pair of \new{$S_k\times S_{N-k}$}-symmetric solutions $U_{ab}(\mu,d)$ and $V_{ab}(\mu,d)$ of \eqref{SteadyLDS} that bifurcate at a fold bifurcation at $\mu = \mu_{l}(d)$ and exist for all $\mu \in [\mu_l(d),\mu_{ab}]$. These solutions are smooth in $(\mu,d)$, and for each fixed $\mu$\new{,} we have $U_{ab}(\mu,d) \to \bar{A}_+^{(k)}(\mu)$ and $V_{ab}(\mu,d) \to \bar{B}^{(k)}(\mu)$ as $d \to 0^+$. The function $\mu_{l}(d)$ satisfies $\mu_{l}(d) = 3\sqrt[3]{\frac{k^2}{4}}d^\frac{2}{3} + \mathcal{O}(d)$.
\item There are constants $d_{cd},\mu_{cd} > 0$ and a smooth function $\mu_l:[0,d_{cd}] \to [0,\mu_{cd}]$ such that for each $d \in (0,d_{cd}]$\new{,} there is a pair of \new{$S_k\times S_{N-k}$}-symmetric solutions $U_{cd}(\mu,d)$ and $V_{cd}(\mu,d)$ of \eqref{SteadyLDS} that bifurcate at a fold bifurcation at $\mu = \mu_{l}(d)$ and exist for all $\mu \in [\mu_l(d),\mu_{cd}]$. These solutions are smooth in $(\mu,d)$, and for each fixed $\mu$\new{,} we have $U_{cd}(\mu,d) \to \bar{C}_+^{(k)}(\mu)$ and $V_{cd}(\mu,d) \to \bar{D}^{(k)}(\mu)$ as $d \to 0^+$. The function $\mu_{l}(d)$ satisfies $\mu_{l}(d) = 3\sqrt[3]{\frac{(N-k)^2}{4}}d^\frac{2}{3} + \mathcal{O}(d)$.
\end{compactenum}
\end{lem}

\begin{proof}
We will only prove the first case in the lemma since the second can simply be obtained by exchanging $v_1$ and $v_2$ in (\ref{all-to-all}) in what follows.  In this notation\new{,} the functions $\bar A_\pm^{(k)}(\mu)$ and $\bar B^{(k)}(\mu)$ correspond to solutions of \eqref{all-to-all} with $d = 0$ as well as $(v_1,v_2) = (u_\pm(\mu),0)$ and $(u_+(\mu),u_-(\mu))$, respectively, for all $\mu \in [0,1]$. For convenience, we will denote
\begin{equation}\label{AllFunction}
	\begin{split}
	\mathcal{F}_1(v_1,v_2,\mu,d) &= d(N-k)(v_2-v_1)+f(v_1,\mu) \\
	\mathcal{F}_2(v_1,v_2,\mu,d) &= dk(v_1-v_2)+f(v_2,\mu)
	\end{split}
\end{equation}
to emphasize the connection \new{of the system \eqref{all-to-all} with the function $\mathcal{F}$ that we} defined in \eqref{Function}. Importantly, solutions of $\mathcal{F}_1 = \mathcal{F}_2 = 0$ represent \new{$S_k\times S_{N-k}$}-invariant steady-state solutions of \eqref{Function}.

In the present scenario, we recall that $u_+(0) = 1$ and restrict ourselves to a neighbourhood of the solution $(v_1,v_2,\mu,d) = (1,0,0,0)$ to \eqref{all-to-all}. From the previous discussion, this represents a \new{$S_k\times S_{N-k}$}-invariant neighbourhood of the solution $(U,\mu,d) = (\new{\bar A_+^{(k)}}(0),0,0)$ to $\mathcal{F}$ in \eqref{Function}. Then, since 
\begin{equation}
	\frac{\partial\mathcal{F}_1}{\partial v_1} (1,0,0,0) = f_u(1,0) \ne 0,
\end{equation} 
the implicit function theorem guarantees the existence of a unique smooth function $v_1^{*}(v_2,\mu,d)$ satisfying $v_1^{*}(0,0,0) =1$ and $\mathcal{F}_1 (v_1^{*}(v_2,\mu,d),v_2,\mu,d) = 0$ for all $(v_2,\mu,d)$ in a neighbourhood of $(0,0,0)$. Hence, the Taylor series of $v_1^{*}(v_2,\mu,d)$ about $(v_2,\mu,d) = (0,0,0)$ is given by  
\begin{equation} \label{all_1_v}
    v_1^{*}(v_2,\mu,d) = 1 + \mathcal{O}(|\mu| + |d| + |v_2||d|).
\end{equation}
The analysis now reduces to solving $\mathcal{F}_2(v_1,v_2,\mu,d) = 0$ with $v_1 = v_1^*(v_2,\mu,d)$ in a neighbourhood of $(v_2,\mu,d) = (0,0,0)$.

As in the previous lemmas of this section, let us introduce the re-scaled variables
\begin{equation} \label{lem3.1scaling}
    \mu = \nu^2, \qquad
    d = \nu^3 \tilde{d}, \qquad
    v_2 = \nu\tilde{v_2}, \qquad
\end{equation}
for $|\nu| \ll 1$, so that \eqref{all_1_v} becomes
\begin{equation}
	v_1^{*}(v_2,\mu,d) = 1 + \mathcal{O}(\nu^2).	
\end{equation}
Then, putting $v_1 = v_1^*(v_2,\mu,d)$ into $\mathcal{F}_2$ and expanding using \eqref{all_1_v} and \eqref{NormalForm1} gives 
\begin{equation} \label{lem3.1_F}
    \mathcal{F}_2(v_1^*(v_2,\mu,d),v_2,\mu,d) = \nu^3(k\tilde{d}-\tilde{v}_2+\tilde{v}_2^3)+\mathcal{O}(\nu^4).
\end{equation}
Hence, solving $\mathcal{F}_2(v_1^*(v_2,\mu,d),v_2,\mu,d) = 0$ is equivalent to solving 
\begin{equation}
	0 = k\tilde{d}-\tilde{v}_2+\tilde{v}_2^3+\mathcal{O}(\nu)
\end{equation}
after dividing off $\nu^3$. It is therefore clear that the above expression experiences a fold bifurcation at 
\begin{equation}
	(\tilde{v}_2,\tilde{d},\nu) = \bigg(\sqrt[3]{\frac{k}{2}},\frac{2}{3\sqrt{3}k},0\bigg),
\end{equation} 
for which we may follow as in the proofs of the previous lemmas to demonstrate the persistence of this fold into sufficiently small $\nu > 0$. This completes the proof.
\end{proof} 

Let us now turn to the bifurcation near $\mu = 1$. The following lemma details the persistence of the connections between $\bar A_-^{(k)}(\mu)$ and $\bar A_+^{(k)}(\mu)$\new{,} and $\bar C_-^{(k)}(\mu)$ and $\bar C_+^{(k)}(\mu)$ into $d > 0$. 

\begin{lem}\label{lem:all_2} 
For any $N \geq 2$ and $k = 1,\dots, \new{\lfloor N/2 \rfloor}$, the following is true for \eqref{SteadyLDS}. 
\begin{compactenum}
\item There exist constants $d_{aa},\mu_{aa} > 0$ and a smooth function $\mu_{r}:[0,d_{aa}] \to [\mu_{aa},1]$ such that for each fixed $d \in (0,d_{aa}]$\new{,} there is a pair of \new{$S_k\times S_{N-k}$}-symmetric solutions $U_{aa}(\mu,d)$ and $V_{aa}(\mu,d)$ of \eqref{SteadyLDS} that bifurcate at a fold bifurcation at $\mu = \mu_{r}(d)$ and exist for all $\mu \in [\mu_{aa},\mu_{r}(d)]$. These solutions are smooth in $(\mu,d)$, and for each fixed $\mu$\new{,} we have $U_{aa}(\mu,d) \to \bar A_-^{(k)}(\mu)$ and $V_{aa}(\mu,d) \to \bar{A}_+^{(k)}(\mu)$ as $d \to 0^+$.
The function $\mu_{r}(d)$ satisfies $\mu_{r}(d) = 1- (N-k)d + \mathcal{O}(d^{\frac{3}{2}})$.
\item There exist constants $d_{cc},\mu_{cc} > 0$ and a smooth function $\mu_{r}:[0,d_{cc}] \to [\mu_{cc},1]$ such that for each fixed $d \in (0,d_{cc}]$\new{,} there is a pair of \new{$S_k\times S_{N-k}$}-symmetric solutions $U_{cc}(\mu,d)$ and $V_{cc}(\mu,d)$ of \eqref{SteadyLDS} that bifurcate at a fold bifurcation at $\mu = \mu_{r}(d)$ and exist for all $\mu \in [\mu_{cc},\mu_{r}(d)]$. These solutions are smooth in $(\mu,d)$, and for each fixed $\mu$\new{,} we have $U_{cc}(\mu,d) \to \bar C_-^{(k)}(\mu)$ and $V_{cc}(\mu,d) \to \bar{C}_+^{(k)}(\mu)$ as $d \to 0^+$.
The function $\mu_{r}(d)$ satisfies $\mu_{r}(d) = 1- kd + \mathcal{O}(d^{\frac{3}{2}})$. 
\end{compactenum}
\end{lem}

\begin{proof}
	Since we are interested in \new{$S_k\times S_{N-k}$}-invariant solutions of \eqref{SteadyLDS}, we may again proceed as in the proof of Lemma~\ref{lem:all_1} and reduce to solving the equations \eqref{all-to-all}. As in Lemma~\ref{lem:all_1}, we will only prove the first statement since the second follows from simply interchanging $v_1$ and $v_2$ in what follows. The difference is that now we focus a neighbourhood of the solution $(v_1,v_2,\mu,d) = (1,0,1,0)$, corresponding to a \new{$S_k\times S_{N-k}$}-invariant neighbourhood of the solution $(U,\mu,d) = (\bar A_\pm^{(k)}(1),1,0)$ of \eqref{SteadyLDS}. Since we have 
	\begin{equation}
		\frac{\partial\mathcal{F}_2}{\partial v_2} (1,0,1,0) = f_u(0,1) \ne 0,
	\end{equation}  
	we may apply the implicit function theorem to obtain a smooth function $v_2^*(v_1,\mu,d)$ defined in a neighbourhood of $(v_1,\mu,d) = (1,1,0)$ satisfying $v_2^*(1,1,0) = 0$ and $\mathcal{F}_2(v_1,v_2^*(v_1,\mu,d),\mu,d) = 0$. Solving 
	\begin{equation}
		\mathcal{F}_1(v_1,v_2^*(v_1,\mu,d),\mu,d)=0
	\end{equation}
	in a neighbourhood of $(v_1,\mu,d) = (1,1,0)$ then proceeds as in the lemmas of the previous subsections and is therefore omitted.
\end{proof} 

\new{Next, we investigate the patterns $\bar{A}^{(k)}_-(\mu)$ and $\bar{C}^{(k)}_-(\mu)$ near $\mu = 0$ for $0<d\ll1$ and prove that they disappear through a collision with the branch of $S_N$-symmetric equilibria at which $v_1=v_2=u_-(\mu)$. We state this lemma for the system \eqref{all-to-all} and refer to the left panel in Figure~\ref{fig:llcorner} for an illustration of the bifurcation diagram}.

\begin{figure}
    \centering
    \includegraphics{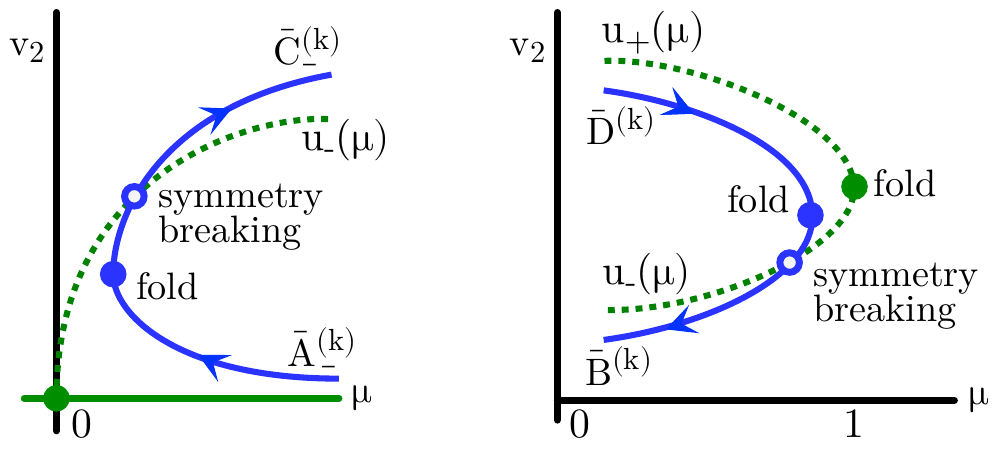}
    \caption{\new{We illustrate the bifurcation diagrams of $S_k\times S_{N-k}$-symmetric solutions for all-to-all coupling for fixed $0<d\ll1$ and $1\leq k\leq\lfloor N/2\rfloor$ outlined in Lemma~\ref{lem:all_ll} for $\mu$ near zero (left panel) and in Lemma~\ref{lem:all_3} for $\mu$ near one (right panel). The dotted curve corresponds to the family of $S_N$-symmetric equilibria (where $v_1=v_2$), and the solid curve reflects the branch of $S_k\times S_{N-k}$-symmetric solutions (with $v_1\neq v_2$), which bifurcates from the branch of $S_N$-symmetric solutions in a symmetry-breaking bifurcation and also undergoes a fold bifurcation; arrows indicate the direction of increasing $\phi$ (left) and $s$ (right) that parametrize these branches.}}
    \label{fig:llcorner}
\end{figure}

\begin{lem}\label{lem:all_ll}
\new{
For each $N \geq 2$, $k = 1,\dots, \new{\lfloor N/2 \rfloor}$, and $\epsilon>0$, there is a constant $\delta>0$ with the following property. There is a unique branch of positive $S_k\times S_{N-k}$-symmetric equilibria of \eqref{SteadyLDS} near $(u,\mu)=(0,0)$, and this branch is smooth and given by}
\[
(v_1,v_2,d,\mu)(\phi,s) = \left(s\cos\phi, s\sin\phi,
\frac{(\cos\phi+\sin\phi)\cos\phi\sin\phi}{k\cos\phi+(N-k)\sin\phi} s^2 + \mathcal{O}(s^4),
\frac{k\cos^3\phi+(N-k)\sin^3\phi}{k\cos\phi+(N-k)\sin\phi} s^2 + \mathcal{O}(s^4)\right)
\]
\new{for $\epsilon<\phi<\frac{\pi}{2} -\epsilon$ and $0\leq s<\delta$. Furthermore, the branch can also be parametrized smoothly by $(d,\phi)$ via}
\[
s = \frac{k\cos\phi+(N-k)\sin\phi}{(\cos\phi+\sin\phi)\cos\phi\sin\phi} d + \mathcal{O}(d^2).
\]
\new{It bifurcates at $\mu=\mu_\mathrm{sb}(d)=\frac{Nd}{2}+\mathcal{O}(d^2)$ from the $S_N$-symmetric branch $v_1=v_2=u_-(\mu)$ (this bifurcation point is the same for each $k$), and it undergoes a unique fold bifurcation at $\mu=\mu_\mathrm{fd}(d)$ for some $0<\phi \le \frac{\pi}{4}$ with $0<\mu_\mathrm{fd}(d) \le \mu_\mathrm{sb}(d)$. As $\phi\to0$ and $\phi\to\frac{\pi}{2}$, the parameter $d$ approaches zero, and the equilibria $(v_1,v_2)$ approach $\bar A^{(k)}_-(\mu(0,s))$ and $\bar{C}^{(k)}_-(\mu(\frac{\pi}{2},s))$, respectively.}
\end{lem}

\begin{proof}
Recall that we need to solve the system \eqref{all-to-all} given by
\begin{subequations} \label{e:aa}
\begin{empheq}[left=\empheqlbrace]{align}
    d(N-k)(v_2-v_1)+f(v_1,\mu) = 0, \label{e:aa1}\\
    dk(v_1-v_2)+f(v_2,\mu) = 0 \label{e:aa2}
\end{empheq}
\end{subequations}
where $v_1$ denotes the values of the vector in $\R^N$ at the first $k$ indices and $v_2$ the values of the last $N-k$ indices. Recall also that
\[
f(u,\mu) = -\mu u + u^3 + \mathcal{O}(\mu^2 u + \mu u^3 + u^5)
\]
First, when $u=v_1=v_2$, equation (\ref{e:aa}) reduces to $f(u,\mu)=0$ which admits the solution branch $(u,\mu)=(u_-(s),\mu_\mathrm{h}(s))$ for $0\leq s\ll1$ with $u_-(0)=\mu_\mathrm{h}(0)=0$ and $\mu_\mathrm{h}(s)>0$ for $s>0$. Next, we subtract the two equations in (\ref{e:aa}) to obtain
\begin{equation}\label{e:a1}
dN(v_2-v_1) + f(v_1,\mu) - f(v_2,\mu) = 0.
\end{equation}
Defining $v=(v_1,v_2)$ and 
\[
g(v_1,v_2,\mu) := \int_0^1 f_u(v_2+\tau(v_1-v_2),\mu) d\tau = -\mu + v_1^2 + v_1v_2 + v_2^2 + \mathcal{O}(\mu^2+|v|^4),
\]
we can write (\ref{e:a1}) as
\[
dN(v_2-v_1) + f(v_1,\mu) - f(v_2,\mu) = dN(v_2-v_1) + (v_1-v_2) g(v_1,v_2,\mu) = 0.
\]
Since we already analysed the case $v_1=v_2$, we can divide by $v_1-v_2$ to arrive at
\[
- dN + g(v_1,v_2,\mu) = -dN - \mu + v_1^2 + v_1v_2 + v_2^2 + \mathcal{O}(\mu^2+|v|^4) = 0,
\]
which we can solve for $\mu$ as a function $\mu(v_1,v_2,d)$ near $(v_1,v_2,d)=(0,0,0)$ where
\[
\mu(v_1,v_2,d) = v_1^2 + v_1v_2 + v_2^2 - dN + \mathcal{O}(d^2+|v|^4).
\]
It therefore suffices to solve (\ref{e:aa2}) which becomes
\begin{eqnarray*}
0 & = & dk (v_1-v_2) + f(v_2,\mu)
\\ & = & 
dk (v_1-v_2) - \mu(v_1,v_2,d) v_2 + v_2^3 + v_2 \mathcal{O}(\mu(v_1,v_2,d)^2 + \mu(v_1,v_2,d) v_2^2 + v_2^4)
\\ & = &
dk (v_1-v_2) + v_2 (dN - v_1^2 - v_1v_2 + \mathcal{O}(d^2+|v|^4))
\end{eqnarray*}
and finally \new{arrives} at
\begin{equation}\label{e:a2}
dk v_1 + d(N-k) v_2 - v_1 v_2 (v_1+v_2) + v_2 \mathcal{O}(d^2+|v|^4) = 0.
\end{equation}
To solve (\ref{e:a2}), we set $v=(v_1,v_2)=s(\cos\phi,\sin\phi)$ with $0\leq s\ll1$ and $0<\phi<\pi/2$ so that (\ref{e:a2}) becomes
\begin{equation}\label{e:aa3}
d (k\cos\phi + (N-k)\sin\phi) - s^2 \cos\phi\sin\phi(\cos\phi+\sin\phi) + \mathcal{O}(d^2+s^4) = 0
\end{equation}
after dividing by $s$. Using the implicit function theorem, we can now solve this equation for $d$ as a function of $(\phi,s)$ with $|s|<\delta$ for some $\delta>0$ uniformly in $0\leq\phi\leq\pi/2$ and obtain
\[
d(\phi,s) = \frac{(\cos\phi+\sin\phi)\cos\phi\sin\phi}{k\cos\phi+(N-k)\sin\phi} s^2 + \mathcal{O}(s^4).
\]
Substituting the expressions for $(v_1,v_2,d)$ as functions of $(\phi,s)$ into the formula for $\mu(v_1,v_2,d)$ gives
\[
\mu(\phi,s) = \frac{k\cos^3\phi+(N-k)\sin^3\phi}{k\cos\phi+(N-k)\sin\phi} s^2 + \mathcal{O}(s^4),
\]
which proves the first part of the lemma. Inspecting the limits $\phi\to0$ and $\phi\to\frac{\pi}{2}$ for each fixed $0<s<\delta$, we see that $d$ converges to zero, $\mu$ approaches $s^2+\mathcal{O}(s^4)$, and $(v_1,v_2)$ approach $\bar A^{(k)}_-(\mu(0,s))$ and $\bar{C}^{(k)}_-(\mu(\frac{\pi}{2},s))$, respectively, as claimed. Similarly, substituting $\phi=\pi/4$ into the expressions for $(v_1,v_2,d,\mu)$ shows that the branch we constructed bifurcates from the branch of $S_N$-symmetric equilibria at $(\mu,d)=(\frac{1}{2},\frac{1}{N})s^2+(\mathcal{O}(s^4),\mathcal{O}(s^4))$.

Fixing $\epsilon > 0$, we now restrict ourselves to values of $\phi$ with $\epsilon < \phi < \frac{\pi}{2}-\epsilon$. For such values, we can solve for $s$ as a function of $(d,\phi)$ where
\[
s^2 = s(d,\phi)^2 = \frac{k\cos\phi+(N-k)\sin\phi}{(\cos\phi+\sin\phi)\cos\phi\sin\phi} d + \mathcal{O}(d^2)
\]
so that $\mu$ can be written as a smooth function of $(d,\phi)$ via
\[
\mu(d,\phi) = \left( \frac{k\cos^3\phi+(N-k)\sin^3\phi}{(\cos\phi+\sin\phi)\cos\phi\sin\phi} + \mathcal{O}(d) \right) d =: (\tilde{\mu}(\phi)+\mathcal{O}(d)) d.
\]
It remains to prove that the branch has a unique fold for each fixed $d$, and it suffices to prove this for $\tilde{\mu}(\phi)$ since its folds are generic and therefore persist upon adding the remainder term in $d$. We now outline the strategy for proving that $\tilde{\mu}(\phi)$ has a unique generic fold. Note that $\tilde{\mu}(\phi)\to\infty$ for $\phi\to0$ and $\phi\to\frac{\pi}{2}$, and it therefore suffices to show that $\tilde{\mu}^\prime(\phi)$ increases strictly, which holds when $\tilde{\mu}^{\prime\prime}(\phi)>0$. The second derivative $\tilde{\mu}^{\prime\prime}(\phi)$ is a fraction of trigonometric polynomials, and we can then show that this derivative is indeed positive (we omit the tedious details). The fold bifurcation occurs in the region $0<\phi \le \frac{\pi}{4}$ since $\tilde{\mu}^\prime(\frac{\pi}{4})=\frac{3}{2}(N-2k)\ge 0$ (recall that $1\leq k\leq\lfloor N/2 \rfloor$).
\color{black}
\end{proof}

Finally, we describe the continuation into \new{$0< d \ll 1$} of the connection between $\bar{B}^{(k)}(\mu)$ and $\bar{D}^{(k)}(\mu)$ near $\mu = 1$. We comment that this lemma covers the connection in the \new{top-right corner of Figure~\ref{fig:Intro}(iii) and that the results are also illustrated in the right panel in Figure~\ref{fig:llcorner}}.

\begin{lem}\label{lem:all_3} 
\new{
For each $N \geq 2$ and $k = 1,\dots, \new{\lfloor N/2 \rfloor}$, there is a constant $\delta>0$ with the following property. There is a unique branch of $S_k\times S_{N-k}$-symmetric equilibria of \eqref{SteadyLDS} near $(u,\mu)=(1,1)$, and this branch is smooth and given by}
\[
(v_1,v_2,\mu)(s,d) = (1+s, 1-s-Nd+\mathcal{O}(s^2+d^2), 1-s^2-(N-k)(2s+Nd)d+\mathcal{O}(ds^2+s^3+d^3))
\]
\new{for $|s|<\delta$ and $0\leq d<\delta$. Furthermore, this branch bifurcates at $\mu=1-\frac{N^2d}{4}+\mathcal{O}(d^3)$ when $s=-\frac{Nd}{2}+\mathcal{O}(d^2)$ from the $S_N$-symmetric branch at which $v_1=v_2=u_-(\mu)$ (this bifurcation point is the same for each $k$) and undergoes a unique fold bifurcation at $\mu=1-k(N-k)d^2+\mathcal{O}(d^3)$ when $s=-(N-k)d+\mathcal{O}(d^2)$. As $d\to0^+$, the solution $(v_1,v_2)$ converges to $\bar B^{(k)}(\mu(s,0))$ for fixed $s>0$ and to $\bar{D}^{(k)}(\mu(s,0))$ for $s<0$.}
\end{lem}

\begin{proof}
We consider the neighbourhood of the solution $(v_1, v_2, \mu, d) = (1,1,1,0)$ to \eqref{all-to-all}. For simplicity, denote \new{$\mu=1-\tilde{\mu}$, $v_1=1+\tilde{v}_1$, and $v_2=1+\tilde{v}_2$}. 
Using the expansion \eqref{NormalForm2} these new variables transform (\ref{all-to-all}) to
\begin{subequations} \label{new_all_to_all}
\begin{empheq}[left=\empheqlbrace]{align}
    \tilde{\mathcal{F}}_1 (\tilde{v}_1,\tilde{v}_2,\tilde{\mu},d) &= d(N-k)(\tilde{v}_2-\tilde{v}_1)+\tilde{\mu}-\tilde{v}_1^2-b\tilde{\mu}\tilde{v}_1+\mathcal{O}(\tilde{\mu}^2+ \new{\tilde{\mu}|\tilde{v}_1|^2+|\tilde{v}_1|^3}) = 0, \label{new_all-to-all_1}\\
    \tilde{\mathcal{F}}_2 (\tilde{v}_1,\tilde{v}_2,\tilde{\mu},d) &= dk(\tilde{v}_1-\tilde{v}_2)+\tilde{\mu}-\tilde{v}_2^2-b\tilde{\mu}\tilde{v}_2+\mathcal{O}(\tilde{\mu}^2+\new{\tilde{\mu}|\tilde{v}_2|^2+|\tilde{v}_2|^3}) = 0. \label{new_all-to-all_2}
\end{empheq}
\end{subequations}
We can apply the implicit function theorem to equation (\ref{new_all-to-all_1}) and find a unique smooth function $\tilde{\mu}^{*}(\tilde{v}_1,\tilde{v}_2,d)$, defined in a neighbourhood of $(\tilde{v}_1,\tilde{v}_2,d) = (0,0,0)$, satisfying $\tilde{\mu}^{*}(0,0,0) = 0$ and $\tilde{\mathcal{F}}_1 (\tilde{v}_1,\tilde{v}_2,\tilde{\mu}^{*} (\tilde{v}_1,\tilde{v}_2,d),d) = 0$, in the neighbourhood of $(0,0,0,0)$. Moreover, we have the expansion
\begin{equation} \label{lem:all_3_mu}
    \tilde{\mu}^{*} (\tilde{v}_1,\tilde{v}_2,d) = -d(N-k)(\tilde{v}_2-\tilde{v}_1) + \tilde{v}_1^2 + \mathcal{O}(|d||\tilde{v}|^2 + |\tilde{v}|^3),
\end{equation}

\new{where $\tilde{v} = (\tilde{v}_1,\tilde{v}_2)$.} Next, we subtract the two equations in (\ref{all-to-all}) to obtain
\begin{equation}\label{lem:all_3_diff}
    Nd(v_2-v_1) \new{+} f(v_1,\mu) - f(v_2,\mu) = 0.
\end{equation}
Defining 
\[
g(\tilde{v}_1,\tilde{v}_2,\tilde{\mu}) := \int_0^1 f_u(v_2+\tau(v_1-v_2),\mu) d\tau = -\tilde{v}_1-\tilde{v}_2 + \mathcal{O}(|\tilde{\mu}|+|\tilde{\mu}||\tilde{v}_1|+|\tilde{\mu}||\tilde{v}_2|+|\tilde{v}|^2),
\]
we can write (\ref{lem:all_3_diff}) in new variables as
\[
Nd(\tilde{v}_2-\tilde{v}_1) + f(1+\tilde{v}_1,1-\tilde{\mu}) - f(1+\tilde{v}_2,1-\tilde{\mu})
= Nd(\tilde{v}_2-\tilde{v}_1) + g(\tilde{v}_1,\tilde{v}_2,\tilde{\mu}) (\tilde{v}_1 - \tilde{v}_2) = 0.
\]
Assuming $\tilde{v}_1 \ne \tilde{v}_2$, we can divide by $\tilde{v}_2 - \tilde{v}_1$ to arrive at
\[
\new{Nd -} g(\tilde{v}_1,\tilde{v}_2,\tilde{\mu}) = \new{Nd} + \tilde{v}_1 + \tilde{v}_2 + \mathcal{O}(|\tilde{\mu}|+|\tilde{\mu}||\tilde{v}_1|+|\tilde{\mu}||\tilde{v}_2|+|\tilde{v}|^2) = 0.
\]
We can further substitute the expression (\ref{lem:all_3_mu}) for $\tilde{\mu}$ and find
\[
\new{Nd} + \tilde{v}_1 + \tilde{v}_2 + \mathcal{O}(d|\tilde{v}| + |\tilde{v}|^2) = 0.
\]
Using the implicit function theorem to solve for $\tilde{v}_2$ as a function of $(\tilde{v}_1,d)$ in a neighbourhood of $(0,0)$ gives
\begin{equation} \label{lem:all_3_v2}
    \tilde{v}_2^{*} (\tilde{v}_1,d) = -Nd - \tilde{v}_1 + \mathcal{O}(d|\tilde{v}_1|+|\tilde{v}_1|^2),
\end{equation}
and substituting this expression into \eqref{lem:all_3_mu} we finally find
\begin{equation} \label{lem:all_3_mu_final}
    \tilde{\mu}^{*} (\tilde{v}_1,d) = \tilde{v}_1^2  + (N-k) d (2\tilde{v}_1 + Nd) + \mathcal{O}(d\tilde{v}_1^2+|\tilde{v}_1|^3 + d^3).
\end{equation}
Hence, a fold bifurcation takes place at $\tilde{v}_1 = \tilde{v}_1^*(d) = -(N-k)d + \mathcal{O}(d^2)$ when $d$ is near $0$, \new{and substituting $\tilde{v}_1^*(d)$ into $\tilde{\mu}^{*} (\tilde{v}_1,d)$ gives the claimed expansion. Solving the equation $\tilde{v}_1=\tilde{v_2}$ for $d$ near zero using the implicit function theorem provides the existence and expansion of the symmetry-breaking bifurcation from the family of homogeneous equilibria.} This completes the proof.  
\end{proof} 


\section{Discussion}\label{sec:Discussion}

In this paper, we characterized the branches of localized steady states of lattice dynamical systems on a ring with $N$ nodes with \new{symmetric sparse, almost all-to-all, and all-to-all coupling}. We found snaking branches with $\frac{N}{2}$ saddle-node bifurcations on each side for sparse coupling with interaction range $m=1,2$ and figure-eight \new{type branches that start and end at homogeneous patterns} for all-to-all coupling. \new{The case of almost all-to-all coupling for even $N$ is more complicated due to additional symmetries, and we provided a detailed analysis only for $N\in\{6,8\}$.}

\begin{figure}
\center
\includegraphics[scale=1]{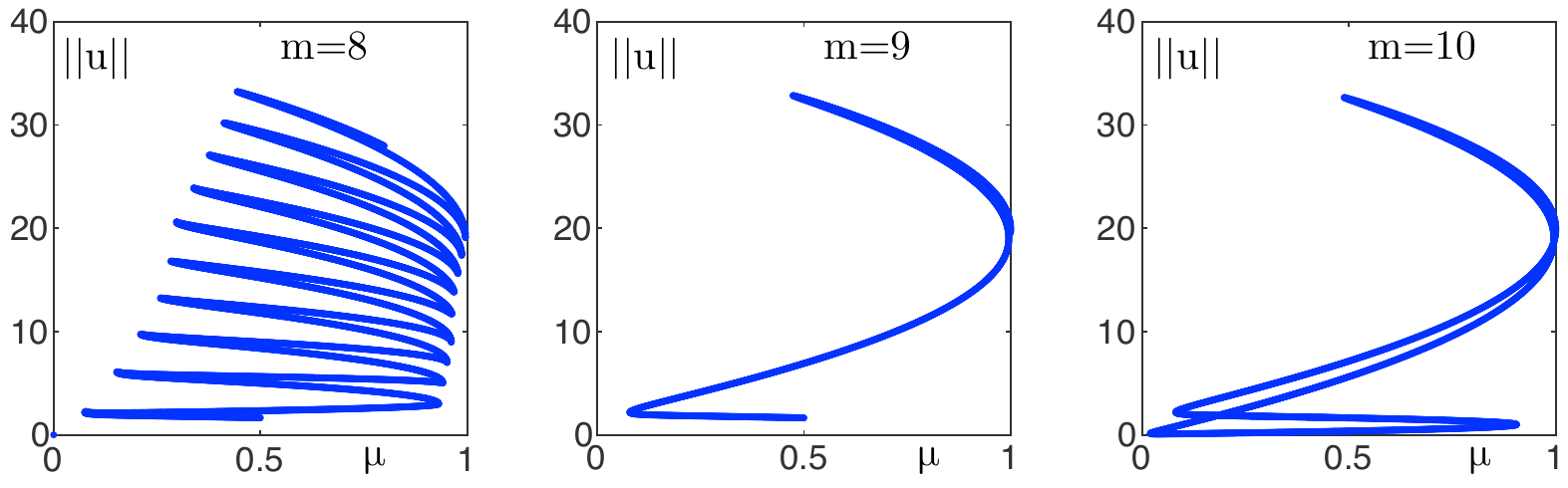}
\caption{Shown are the results of numerical continuation for $d=0.005$ and $N=20$ nodes with (from left to right) sparse coupling for $m=8$, almost all-to-all coupling $m=9$, and all-to-all coupling $m=10$. Each branch starts at $\mu=0.5$ with a single node at the upper stable branch of the zero set of $f(u,\mu)=-\mu u+2u^3-u^5$ and the remaining nodes at zero.}
\label{fig:Exceptions}
\end{figure} 

\new{There are many questions we did not address in this paper, and we now outline a few of these. First, for sparse coupling, we do not know whether solution branches terminate at a branch of homogeneous equilibria or pass near those branches. Numerical continuation turns out to be complicated for small coupling strengths since the Jacobian has $N$ eigenvalues close to zero near the homogeneous branches for $0<d\ll1$, which makes continuation difficult and allows for the possibility that numerical solutions may jump onto different branches.}

\new{Next, we will comment on different forms of coupling. Note first that, near the anti-continuum limit, the precise form of coupling matters only near the pitchfork at $\mu=0$ and the fold at $\mu=1$: the branches for $0<\mu<1$ can always be continued into $0<d\ll1$ regardless of the coupling function, so the key task is to track these branches near $\mu=0,1$. We believe that our results for symmetric nearest-neighbour and next-nearest-neighbour coupling can be extended to other values of $m$: we conjecture that we will always find snaking curves that exhibit $\lfloor\frac N 2\rfloor$ saddle nodes as long as the interaction range $m$ is strictly less than $\lfloor\frac N 2\rfloor-1$, and we expect that the methods we utilized here can be used to study this more general case. Figure~\ref{fig:Exceptions} indicates that the case of almost all-to-all coupling (when $m=\lfloor\frac N 2\rfloor-1$) is more complicated. The case of asymmetric or distance-dependent coupling coefficients is also unclear. In the case of asymmetric coupling, our methods should be applicable, though we expect that the specific details will be different. Distance-dependent coupling is likely more complicated: for instance, the case of all-to-all coupling will be different since the $S_N$ symmetry is no longer present in this case.}

\new{Finally, we briefly discuss possible extensions to more complicated finite lattices. Surprisingly, snaking was observed also on finite random graphs; see \cite{McCullen}. The formal analysis in \cite{Susanto2} and the complementary rigorous results in \cite{Bramburger2} show that the shape of branches near folds seems to be determined by a small number of motifs where only a few nodes interact, whilst the states of all other nodes are essentially irrelevant. It might be this underlying generic structure that leads to the universality of snaking observed on random graphs in \cite{McCullen}. Using computations to obtain summary statistics of snaking on random graphs would be a feasible first step towards a more general description of snaking and branch structures on general finite graphs.}


\begin{Acknowledgment}
\new{We are grateful to the referees for helpful comments and suggestions. Sandstede was partially supported by the NSF under grant DMS-1714429.}
\end{Acknowledgment}


\bibliography{main}
\bibliographystyle{sandstede}

\end{document}